\documentclass[12pt]{amsart}
\usepackage{amscd}
\usepackage[all]{xy}
\usepackage{pifont}
\usepackage{tipa}
\usepackage{bbm}
\usepackage{amssymb}
\usepackage{amsfonts}
\usepackage{mathrsfs}
\usepackage{amsmath}
\usepackage{geometry}
\theoremstyle{definition}\theoremstyle{proposition}
\newtheorem{theorem}{Theorem}[section]
\newtheorem{proposition}[theorem]{Proposition}

\theoremstyle{definition}
\theoremstyle{corollary}\theoremstyle{theorem}
\newtheorem{corollary}[theorem]{Corollary}
\newtheorem{definition}[theorem]{Definition}
\newtheorem{remark}[theorem]{Remark}
\newtheorem{lemma}[theorem]{Lemma}
\numberwithin{equation}{section}
\theoremstyle{definition} \theoremstyle{Remark}
\theoremstyle{Theorem} \theoremstyle{Propsition}
\theoremstyle{Corollary} \numberwithin{equation}{section}
\DeclareMathOperator{\gldim}{\ensuremath{gldim}}
\DeclareMathOperator{\id}{\ensuremath{id}}
\DeclareMathOperator{\injdim}{\ensuremath{injdim}}
\DeclareMathOperator{\Tdeg}{\ensuremath{Tdeg}}
\DeclareMathOperator{\Ext}{\ensuremath{Ext}}
\DeclareMathOperator{\Hom}{\ensuremath{Hom}}
\DeclareMathOperator{\Ker}{\ensuremath{Ker}}
\DeclareMathOperator{\im}{\ensuremath{Im}}

\newcommand{\blankbox}[2]

\begin{document}
\title{Double loop quantum enveloping algebras}

\author{Wu Zhixiang}
\address{ Department of Mathematics, Zhejiang University,
Hangzhou, 310027, P.R.China} \email{wzx@zju.edu.cn}
\thanks{The author is sponsored by NNSF No.11171296, ZJNSF No. Y6100148, Y610027 and Education Department of Zhejiang
Province No. 201019063.}





\begin{abstract}In this paper we describe certain homological
properties and representations of a two-parameter quantum enveloping
algebra $U_{g,h}$ of ${\frak {sl}}(2)$, where $g,h$ are group-like
elements.
\end{abstract}

\maketitle


\section{Introduction}
It is well-known that there is a bijective map $L\rightarrow P_L$
from the set of all oriented links $L$ in $\mathbb{R}^3$ to the ring
$\mathbb{Z}[g^{\pm1},h^{\pm1}]$ of two-variable Laurent polynomials.
$P_L$ is called the Jones-Conway polynomial of the link $L$. The
Jones-Conway polynomial is an isotopy invariant of oriented links
satisfying what knot theorists call ``skein relations" (see [11]).
Suppose $\mathbbm{K}$ is a field with characteristic zero and $q$ is
a nonzero element in $\mathbbm{K}$ satisfying $q^2\neq $ $ 1$. Let
$U_q(\mathfrak{sl}(2))$ be the usual quantum enveloping algebra of
the Lie algebra $\mathfrak{sl}(2)$ with generators $E,F,K^{\pm1}$.
Then the vector space
$$U_{g,h}:=\mathbbm{K}[g^{\pm1},h^{\pm1}]\otimes_{\mathbbm{K}}
U_q(\mathfrak{sl}(2))$$ has been endowed with a Hopf algebra
structure in [13].

We abuse notation and write $g^{\pm1},h^{\pm1},E,F,K^{\pm1}$ for
$g^{\pm1}\otimes 1$, $h^{\pm1}\otimes 1$, $1\otimes E$, $1\otimes
F$, $1\otimes K^{\pm1}$ respectively. In addition,
$g^{+1},h^{+1},K^{+1}$ are abbreviated to $g,h,K$ respectively. Then
$U_{g,h}$ is an algebra over $\mathbbm{K}$ generated by $g,$ $
g^{-1},$ $h,$ $h^{-1},$ $E,$ $F,$ $K,$ $K^{-1}$. These generators
satisfy the following relations.
\begin{align}K^{-1}K=KK^{-1}=1,\quad g^{-1}g=gg^{-1}=1,\quad h^{-1}h=hh^{-1}=1,
\end{align}
\begin{align} KEK^{-1}=q^2E,\quad  gh=hg,\quad  gK=Kg,\quad  gE=Eg,\quad  hE=Eh,\end{align}
\begin{align}   KFK^{-1}=q^{-2}F,\quad  hK=Kh,\quad  hF=Fh,\quad  gF=Fg,\end{align}
\begin{align} EF-FE=\frac{K-K^{-1}g^2}{q-q^{-1}}.\end{align}
 The other operations of the Hopf algebra $U_{g,h}$ are defined as
follows:
\begin{align} \Delta(E)=h^{-1}\otimes E+E\otimes hK,\end{align}
\begin{align} \Delta(F)&=K^{-1}hg^2\otimes F+F\otimes h^{-1},\end{align}
\begin{align} \Delta(K)=K\otimes K,\quad \Delta(K^{-1})=K^{-1}\otimes K^{-1},\end{align}
\begin{align} \Delta(a)=a\otimes a,\quad a\in G, \end{align}
where $G=\{g^mh^n|m,n\in\mathbb{Z}\}$,
\begin{align} \varepsilon(K)=\varepsilon(K^{-1})=\varepsilon(a)=1,\quad a\in G,\end{align}
\begin{align} \varepsilon(E)=\varepsilon(F)=0,\end{align}
and
\begin{align}S(E)=-EK^{-1},\quad S(F)=-KFg^{-2},\end{align}
\begin{align}S(a)=a^{-1},\quad a\in G,\quad S(K)=K^{-1},\quad S(K^{-1})=K.\end{align}

The Hopf algebra $U_{g,h}$ is a special case of the Hopf algebras
defined in [14]. It is isomorphic to the tensor product of
$U_q(\mathfrak{sl}(2))$ and $\mathbb{K}[g^{\pm1},h^{\pm1}]$ as
algebras. However, the coproduct of $U_{g,h}$ is not the usual
coproduct of the tensor product of two coalgebras. Neither is the
antipode.

Homological methods have been used to study Hopf algebras by many
authors (see [2], [15] and their references). However, there are few
examples of Hopf algebras satisfying a given set of homological
properties. In this paper, we describe certain homological
properties of the Hopf algebra $U_{g,h}$ and consequently give an
example satisfying some homological properties. Moreover, we study
the representation theory of the algebra $U_{g,h}$. Similar to [7]
and [8], we can define some version of the Bernstein-Gelfand-Gelfand
(abbreviated as BGG) category $\mathcal {O}$. Furthermore, we
decompose the BGG category $\mathcal {O}$ into a direct sum of
subcategories, which are equivalent to categories of finitely
generated modules over some finite-dimensional algebras.

Let us outline the structure of this paper. In Section 2, we study
the homological properties of $U_{g,h}$. We prove that $U_{g,h}$ is
Auslander-regular and Cohen-Macaulay, and the global dimension and
Gelfand-Kirillov dimension of $U_{g,h}$ are equal.  We also prove
that the center of $U_{g,h}$ is equal to
$\mathbbm{K}[g^{\pm1},h^{\pm1},C]$, where $C$ is the Casimir element
of $U_{g,h}$. To study the category  $\mathcal {O}$ in Section 4, we
prove that $U_{g,h}$ has an anti-involution that acts as the
identity on all of $\mathbb{K}[K^{\pm1},g^{\pm1},h^{\pm1}]$.

Since there is a finite-dimensional non-semisimple module over the
algebra $\mathbb{K}[g^{\pm1},h^{\pm1}]$, there is a
finite-dimensional non-semisimple module over $U_{g,h}$. In Section
3, we compute the extension group $\Ext^1(M,M')$ in the case that
the nonzero $q$ is not a root of unity, where $M,M'$ are
finite-dimensional simple modules over $U_{g,h}$. We prove that the
tensor functor $V\otimes-$ determines an isomorphism from
$\Ext^1(\mathbb{K}_{\alpha',\beta'},\mathbb{K}_{\alpha,\beta})$ to
$\Ext^1(V\otimes\mathbb{K}_{\alpha',\beta'},V\otimes\mathbb{K}_{\alpha,\beta})$
for any finite-dimensional simple $U_{q}(\mathfrak{sl}(2))$-module
$V$. We also obtain a decomposition theory about the tensor product
of two simple $U_{g,h}$-modules. From this, we obtain a Hopf
subalgebra of the finite dual Hopf algebra $U_{g,h}^{\circ}$ of
$U_{g,h}$, which is generated by coordinate functions of
finite-dimensional simple modules of $U_{g,h}$.

In Section 4, we briefly discuss the Verma modules of $U_{g,h}$. The
BGG subcategory $\mathcal{O}$ of the category of representations of
$U_{g,h}$ is introduced and studied. The main results in [8]  also
hold in the category $\mathcal {O}$ over the algebra $U_{g,h}$.

Throughout this paper $\mathbbm{K}$ is a fixed algebraically closed
field with characteristic zero; $\mathbb{N}$  is the set of natural
numbers; $\mathbb{Z}$ is the set of all integers. $*^{+1}$ is
usually abbreviated to $*$. All modules over a ring $R$ are left
$R$-modules.

It is worth mentioning that some results of this article are also
true if $\mathbb{K}$ is not an algebraically closed field. We always
assume that $\mathbb{K}$ is an algebraically closed field for
simplicity throughout this paper.

\noindent
\subsection*{ Acknowledgment} The author would like to thank the
referee  for carefully reading earlier versions of this paper. His
helpful comments and illuminating suggestions have greatly improved
the final version. In particular, the main idea of the proofs of
Theorem 2.1 and Theorem 3.10 was provided by the referee.

\section{some properties of  $U_{g,h}$}
 In this section, we firstly prove that $U_{g,h}$ is a Noetherian domain with a PBW basis. Then we compute the global dimension and
Gelfand-Kirillov dimension of $U_{g,h}$. Moreover, we show that
$U_{g,h}$ is Auslander regular, Auslander Gorenstein, Cohen-Macaulay
and $\Tdeg$-stable. For the undefined terms in this section, we
refer the reader to [2] and [3].

\begin{theorem}[PBW Theorem] The algebra $U_{g,h}$ is a Noetherian domain.
Moreover, it has a PBW basis $\{F^lK^mg^nh^sE^t|l,t\in
\mathbb{Z}_{\geq 0};m,n,s\in \mathbb{Z}\}$.
\end{theorem}
\begin{proof}Let $R=\mathbbm{K}[K^{\pm1},g^{\pm1},h^{\pm1}]$. Since $R$ is a homomorphic image of the polynomial ring $\mathbbm{K}[x_1,x_2,\cdots,x_5,x_6]$ ($\varphi(x_1)=K$, $\varphi(x_2)=K^{-1}$, $\varphi(x_3)=g$,
$\varphi(x_4)=g^{-1}$, $\varphi(x_5)=h$, $\varphi(x_6)=h^{-1}$),
 $R$ is a Noetherian ring
with a basis $\{K^mg^nh^s|m,n,s$ $\in $ $\mathbb{Z}\}$. It is easy
to prove that $R$ is a domain.

Define $\sigma(K^cg^ah^b)=q^{2c}K^cg^ah^b$, $\forall
a,b,c\in\mathbb{Z}$, $\delta(R)\equiv 0$, and  extend $\sigma$ by
additivity and multiplicativity. It is trivial to check that
$\sigma$ is a ring automorphism of $R$, and $\delta:R\rightarrow R$
is a $\sigma$-skew derivation. Hence $R':=R[F;\sigma,\delta]$ is a
Noetherian domain with a basis $\{K^ag^bh^cF^d|a,b,c\in \mathbb{Z},
d\in$ $\mathbb{Z}_{\geq 0}\}$ by [9, Theorem 1.2.9].

Next, define $\sigma'$ on $R'$ via:
$$
\sigma'(K^ag^bh^cF^d)=q^{-2a}K^ag^bh^cF^d,$$(for all integers $d\geq
0$, and $a,b,c\in \mathbb{Z}$), and extend $\sigma'$ by additivity
and multiplicativity. One can check that $\sigma'$ is indeed a ring
automorphism of $R'$. Define $\delta'$ on $R'$ via
$$\delta'(R)\equiv0,\qquad
\delta'(F)=\frac{K-K^{-1}g^2}{q-q^{-1}}.$$Also extend $\delta'$ to
all of $R'$ by additivity and the following equation:
$$\delta'(ab):=\delta'(a)b+\sigma'(a)\delta'(b),\qquad \forall a,b\in R'.$$
One can check that $\delta'$ is a $\sigma'$-skew derivation of $R'$.
Now by the above results, $U_{g,h}$ $=R'[E;\sigma',\delta']$ is
indeed a Noetherian domain, since $R'$ is. Moreover, $U_{g,h}$ has a
basis $\{K^ag^bh^cF^dE^t|a,b,c\in \mathbb{Z}, d,t\in$
$\mathbb{Z}_{\geq 0}\}$. Since
$K^ag^bh^cF^dE^t=q^{-2ad}F^dK^ag^bh^cE^t$,
$$\{F^lK^mg^n
h^sE^t|l,t\in \mathbb{Z}_{\geq 0};m,n,s\in \mathbb{Z}\}$$ is also a
basis of $U_{g,h}$. This basis is called a PBW basis.
\end{proof}

\begin{proposition} (1) $U_{g,h}$ is isomorphic to
$\mathbbm{K}[g^{\pm1},h^{\pm1}]\otimes U_q(\mathfrak{sl}(2))$ as
algebras;

(2) $U_{g,h}$ is an Auslander regular, Auslander Gorenstein and
$\Tdeg$-stable algebra with Gelfand-Kirillov dimension 5.
\end{proposition}

\begin{proof} Define $E':=g^{-1}E$, $K':=g^{-1}K$.  By
Theorem 2.1, $$\{F^aK'^{b}g^ch^dE'^{t}|b,c,d\in \mathbb{Z},a,t\in
\mathbb{Z}_{\geq0}\}$$ is also a basis of $U_{g,h}$. Let
$\varphi(E')=1\otimes E$, $\varphi(F)=1\otimes F$,
$\varphi(K')=1\otimes K$, $\varphi(g)=g\otimes 1$,
$\varphi(h)=h\otimes 1$ and $\varphi$ extends by additivity and
multiplicativity. One can check that $\varphi$  is an epimorphism of
algebras from $U_{g,h}$ to $\mathbbm{K}[g^{\pm1},h^{\pm1}]\otimes$ $
U_g(\mathfrak{sl}(2))$. Similarly, define $\phi(1\otimes E)$ $=$
$E', \phi(1\otimes F)=F, $ $\phi(1\otimes K)$ $=K',$ $ \phi
(g\otimes 1)$ $=g,$ $\phi(h\otimes 1)=h$, and  extend $\phi$ by
additivity and multiplicativity. Then $\phi$ is an epimorphism of
algebras from $\mathbbm{K}[g^{\pm1},h^{\pm1}]\otimes$ $
U_q(\mathfrak{sl}(2))$ to $U_{g,h}$. It is easy to verify that
$\phi\circ\varphi=\id$ and $\varphi\circ\phi=\id$. So $\varphi$ is
an isomorphism of algebras.

Let us recall that if the global dimension of a Noetherian ring $A$,
denoted by $\gldim(A)$, is finite, then $\gldim(A)=\injdim(A)$, the
injective dimension of $A$. From [9, Section 7.1.11], one obtains
that the right global dimension of a Noetherian algebra $A$ is equal
to $\gldim(A)$ as well. In [1], H. Bass proved that if $A$ is a
commutative Noetherian ring with a finite injective dimension, then
$A$ is Auslander-Gorenstein. Thus
$\gldim({\mathbbm{K}}[g^{\pm1},h^{\pm1},K^{\pm1}])$ $=3$, and
$$\gldim U_{g,h}\leq
\gldim({\mathbbm{K}}[g^{\pm1},h^{\pm1},K^{\pm1}])+2=5$$ by [9,
Theorem 7.5.3]. Hence $U_{g,h}$ is an Auslander regular and
Auslander Gorenstein ring by [3, Theorem 4.2].

Recall that an algebra $A$ with total quotient algebra $Q(A)$ is
said to be $\Tdeg$-stable if
$$\Tdeg(Q(A))=\Tdeg(A)=\text{GKdim}(A),$$ where $\text{GKdim}(A)$ is
the Gelfand-Kirillov dimension of $A$. By [15, Example 7.1],
$U_q(\mathfrak{sl}(2))$ is $\Tdeg$-stable, and
$\text{GKdim}(U_q(\mathfrak{sl}(2)))=3$. Since
$$U_{g,h}\cong \mathbbm{K}[g^{\pm1},h^{\pm1}]\otimes
U_q(\mathfrak{sl}(2))\cong
U_q(\mathfrak{sl}(2))[g,g^{-1}][h,h^{-1}],$$
$$\text{GKdim}(U_{g,h})=2+\text{GKdim}(U_q(\frak{sl}(2))=5,$$
and $U_{g,h}$ is $\Tdeg$-stable by [15, Theorem 1.1].
\end{proof}
\begin{remark}(1) Since $U_{g,h}\cong {\mathbbm{K}}[g^{\pm1},h^{\pm1}]\otimes U_q(\frak{sl}(2))$ as algebras, we call the Hopf
algebra $U_{g,h}$ a double loop quantum enveloping algebra.

(2) Since ${\mathbbm{K}}[g^{\pm1},h^{\pm1}]$ and $U_q(\frak{sl}(2))$
are Hopf algebras, ${\mathbbm{K}}[g^{\pm1},h^{\pm1}]\otimes
U_q(\frak{sl}(2))$ has a natural Hopf algebra structure. However, as
$$\Delta(E')=h^{-1}g^{-1}\otimes E'+E'\otimes hK',$$ and $$\Delta(F)=K'^{-1}hg\otimes F+F\otimes h^{-1},$$
by (1.5) and (1.6), the above isomorphism of algebras is not an
isomorphism of Hopf algebras, i.e., $U_{g,h}$ has a different
coproduct than the usual coproduct of the tensor product of the two
coalgebras.
\end{remark}
\begin{corollary}Suppose $q$ is not a root of unity. Then the center of $U_{g,h}$ is equal to
${\mathbbm{K}}[g^{\pm1},h^{\pm1},C]$, where
$C=FE+\frac{qK+q^{-1}K^{-1}g^2}{(q-q^{-1})^2}$.
\end{corollary}
\begin{proof}
Let
$c^{\prime}=F^{\prime}E^{\prime}+\frac{qK^{\prime}+q^{-1}K^{{\prime}-1}}{(q-q^{-1})^2}$,
where $K^{\prime\pm},E^{\prime},F^{\prime}$ are the Chevalley
generators of $U_q(\frak{sl}(2))$. Then the center of
$U_q(\frak{sl}(2))$ is generated by $c^{\prime}$ by [6, Theorem
VI.4.8]. Since $$U_{g,h}\cong \mathbbm{K}[g^{\pm1},h^{\pm1}]\otimes
U_q(\mathfrak{sl}(2))
$$  as algebras by Proposition 2.2,
the center of $U_{g,h}$ is isomorphic to
$\mathbbm{K}[g^{\pm1},h^{\pm1}]\otimes \mathbb{K}[c']$. So the
center of $U_{g,h}$ is equal to
${\mathbbm{K}}[g^{\pm1},h^{\pm1},c_1]$, where
$c_1=g^{-1}FE+\frac{g^{-1}qK+gq^{-1}K^{-1}}{q-q^{-1}}$. Hence the
center of $U_{g,h}$ is equal to
${\mathbbm{K}}[g^{\pm1},h^{\pm1},C]$, where
$C=FE+\frac{qK+q^{-1}K^{-1}g^2}{(q-q^{-1})^2}$.
\end{proof}
The element $C=FE+\frac{qK+q^{-1}K^{-1}g^2}{(q-q^{-1})^2}$ is called
a Casimir element of $U_{g,h}$.

\begin{proposition} There exists an anti-involution $i$ of $U_{g,h}$ that acts as the identity on all of $\mathbb{K}[K^{\pm1},g^{\pm1},h^{\pm1}]$.
\end{proposition}

\begin{proof} Let $i(E)=-KF$, $i(F)=-EK^{-1}$, $i(K^{\pm1})=K^{\pm1}$, $i(g^{\pm1})=g^{\pm1}$, $i(h^{\pm1})=h^{\pm1}$. Extend $i$
by additivity and multiplicativity. Then $i$ is an anti-involution
of $U_{g,h}$, which acts as the identity on all of
$\mathbb{K}[K^{\pm1},g^{\pm1},h^{\pm1}]$.
\end{proof}

Suppose $M$ is a finitely generated module over an algebra $A$. Then
the grade of $M$, denoted by $j(M)$, is defined to be
$$j(M):=\min\{j\geq 0|\Ext^j_A(M,A)\neq 0\}.$$Recall that an algebra
$A$ is Cohen-Macaulay if $$j(M)+\text{GKdim}(M)=\text{GKdim}(A)$$
for every nonzero finitely generated $A$-module $M$.
\begin{proposition} The algebra $U_{g,h}$ is a Cohen-Macaulay algebra with
$\gldim U_{g,h}=5$.\end{proposition}

\begin{proof}Let $A={\mathbbm{K}}[g,h,u,v,K,L][F;\alpha][E;\alpha,\delta]$, where $\alpha|_{{\mathbbm{K}}[g,h,u,v]}=\id$,
$\alpha(K)=$ $q^2K,$ $\alpha(L)=q^{-2}L$, $\alpha(F)=F$,
$\delta(F)=\frac{K-L}{q-q^{-1}}$, and
$\delta({\mathbbm{K}}[g,h,u,v,K,L])=0$. Then $A$ is
Auslander-regular and Cohen-Macaulay by [2, Lemma II.9.10]. Since
$$U_{g,h}\cong A/(gu-1,hv-1,KL-1),$$ $U_{g,h}$ is Auslander-Gorenstein and Cohen-Macaulay
by [2, Lemma II.9.11]. Let ${\mathbbm{K}}$ be the trivial
$U_{g,h}$-module defined by $a\cdot 1=\varepsilon(a)1$. Then
$\text{GKdim}(\mathbbm{K})=0$ and $\gldim(U_{g,h})=5$ by [2,
Exercise II.9.D].\end{proof}

In the presentation for $U_{g,h}$ given in Section 1, the generators
$K^{\pm 1}$, and the generators $E,F$ play a different role
respectively. Similar to [4], we write down an equitable
presentation for $U_{g,h}$ as follows.
\begin{theorem}The algebra $U_{g,h}$ is isomorphic to the unital associative $\mathbbm{K}$-algebra
with generators $x^{\pm 1}$, $y,z$; $u^{\pm1},v^{\pm1}$ and the
following relations:

\begin{eqnarray}x^{-1}x=xx^{-1}=1,\quad u^{-1}u=uu^{-1}=1,\quad
v^{-1}v=vv^{-1}=1,
\end{eqnarray}
\begin{eqnarray} ux=xu,\quad uy=yu,\quad uz=zu,\quad uv=vu,\end{eqnarray}
\begin{eqnarray}   vx=xv,\quad yv=vy,\quad zv=vz,\end{eqnarray}
\begin{eqnarray} \frac{qxy-q^{-1}yx}{q-q^{-1}}=1,\end{eqnarray}
\begin{eqnarray} \frac{qzx-q^{-1}xz}{q-q^{-1}}=1,\end{eqnarray}
\begin{eqnarray} \frac{qyz-q^{-1}zy}{q-q^{-1}}=1.\end{eqnarray}
\end{theorem}
\begin{proof}Let $\mathscr{U}_{u,v}$ be the algebra generated by $x^{\pm1}$, $y$, $z$, $u^{\pm1}$, $v^{\pm1}$ satisfying the relations
from (2.1) to (2.6). Let us define
$\Phi(x^{\pm1})=g^{\mp1}K^{\pm1}$, $\Phi(y)=K^{-1}g+F(q-q^{-1})$,
$\Phi(z)$ $=$ $K^{-1}g-K^{-1}Eq(q-q^{-1})$,
$\Phi(u^{\pm1})=g^{\mp1}$, $\Phi(v^{\pm1})=h^{\pm1}$, and extend
$\Phi$ by additivity and multiplicativity. Then $\Phi$ is a
homomorphism of algebras from $\mathscr{U}_{u,v}$ to $U_{g,h}$.

Define $\Psi(K^{\pm1})=u^{\mp1}x^{\pm1}$,
$\Psi(F)=\frac{y-x^{-1}}{q-q^{-1}}$,
$\Psi(E)=\frac{1-xz}{(q-q^{-1})qu}$, $\Psi(g)=u^{-1}$, and
$\Psi(h)=v$. We extend $\Psi$ by additivity and multiplicativity. It
is routine to check that $\Psi$ is a homomorphism of algebras from
$U_{g,h}$ to $\mathscr{U}_{u,v}$. Since $\Phi\Psi$  fixes each of
the generators $E, F,K^{\pm1}, g^{\pm1}, h^{\pm1}$ of $U_{g,h}$,
$\Phi\Psi=\id$. Similarly we can check that $\Psi\Phi=\id$. So
$\Phi$ is the inverse of $\Psi$.
\end{proof}

Since $\mathscr{U}_{u,v}$ is isomorphic to $U_{g,h}$ as algebras, we
can regard $U_{g,h}$ as an algebra generated by
$x^{\pm1},u^{\pm1},v^{\pm1}$, $y$ and $z$ with relations
(2.1)--(2.6). To make the above algebra isomorphisms $\Phi,\Psi$
into isomorphisms of Hopf algebras, we only need to define the other
operations of the Hopf algebra $U_{g,h}$ with these new generators
as follows:

\begin{eqnarray} \Delta(x^{\pm1})=x^{\pm1}\otimes x^{\pm1},\end{eqnarray}
\begin{eqnarray} \Delta(u^{\pm1})=u^{\pm1}\otimes u^{\pm1},\end{eqnarray}
\begin{eqnarray} \Delta(v^{\pm1})=v^{\pm1}\otimes v^{\pm1},\end{eqnarray}
\begin{eqnarray} \Delta(y)&=x^{-1}\otimes (x^{-1}-v^{-1})+u^{-1}vx^{-1}\otimes (y-x^{-1})+y\otimes v^{-1},\end{eqnarray}
\begin{eqnarray} \Delta(z)&=x^{-1}\otimes x^{-1}+uv^{-1}x^{-1}\otimes (z-x^{-1})+(z-x^{-1})\otimes v,\end{eqnarray}
\begin{eqnarray} \varepsilon(x^{\pm1})=\varepsilon(u^{\pm1})=\varepsilon(v^{\pm1})=1,\end{eqnarray}
\begin{eqnarray} \varepsilon(y)=\varepsilon(z)=1,\end{eqnarray}
and
\begin{eqnarray}S(x^{\pm})=x^{\mp1},\qquad S(u^{\pm1})=u^{\mp1},\qquad S(v^{\pm1})=v^{\mp1},\end{eqnarray}
\begin{eqnarray}S(y)=x-x^{-1}y+u,\qquad  S(z)=x+u^{-1}-u^{-1}xz.\end{eqnarray}
Then one can check that the above isomorphisms $\Phi,\Psi$ are
isomorphisms of Hopf algebras. For example,
$\Delta(\Psi(g^{\mp1}K^{\pm1}))=\Delta(x^{\pm1})=x^{\pm1}\otimes
x^{\pm1}=(\Psi\otimes\Psi)\Delta(g^{\mp1}K^{\pm1})$.

\section{Finite-dimensional representations of $U_{g,h}$}
Let $q$ be a nonzero element in an algebraically closed field
${\mathbbm{K}}$ with characteristic zero. Moreover, we assume that
$q$ is not a root of unity. The main purpose of this section is to
classify all extensions between two finite-dimensional simple
$U_{g,h}$-modules. Let us start with  a description of the
finite-dimensional simple $U_{g,h}$-modules.

For any three elements
$\lambda,\alpha,\beta\in{\mathbbm{K}}^{\times}(=\mathbbm{K}\setminus
\{0\})$ and any $U_{g,h}$-module $V$, let
$$V^{\lambda,\alpha,\beta}=\{v\in V|Kv=\lambda v,gv=\alpha
v,hv=\beta v\}.$$ The $(\lambda,\alpha,\beta)$ is called a weight of
$V$ if $V^{\lambda,\alpha,\beta}\neq 0$. A nonzero vector in
$V^{\lambda,\alpha,\beta}$ is called a weight vector with weight
$(\lambda,\alpha,\beta)$.

The next result is proved by a standard argument.

\begin{lemma} We have $EV^{\lambda,\alpha,\beta}\subseteq
V^{q^2\lambda,\alpha,\beta}$ and $FV^{\lambda,\alpha,\beta}\subseteq
V^{q^{-2}\lambda,\alpha,\beta}$.
\end{lemma}

\begin{definition} Let $V$ be a $U_{g,h}$-module and $(\lambda,\alpha,\beta)\in{\mathbbm{K}}^{\times 3}$.
A nonzero vector $v$ of $V$ is a highest weight vector of weight
$(\lambda,\alpha,\beta)$ if $$Ev=0,\quad Kv=\lambda v, \quad
gv=\alpha v,\quad hv=\beta v.$$ A $U_{g,h}$-module $V$ is a standard
cyclic module with highest weight $(\lambda,\alpha,\beta)$ if it is
generated by a highest weight vector $v$ of weight
$(\lambda,\alpha,\beta)$.
\end{definition}

\begin{proposition} Any nonzero finite-dimensional $U_{g,h}$-module contains a highest
weight vector. Moreover, the endomorphisms induced by $E$ and $F$
are nilpotent.
\end{proposition}
\begin{proof}
By Lie's theorem, there is a nonzero vector $w\in V$ and
$(\mu,\alpha,\beta)$ $\in \mathbbm{K}^{\times 3}$ such that
$$Kw=\mu w,\qquad gw=\alpha w,\qquad hw=\beta w.$$
In fact, there is an elementary and more direct proof as follows.
Since ${\mathbbm{K}}$ is algebraically closed and $V$ is
finite-dimensional, there is a nonzero vector $v\in V$ such that
$Kv=\mu v$ for some element $\mu\in \mathbb{K}$. Moreover $\mu\in
\mathbb{K}^{\times}$ as $K$ is invertible. Let $$V_{\mu}=\{v\in
V|Kv=\mu v\}\not=0.$$ Then $V_{\mu}$ is also a finite-dimensional
vector space. For any $v\in V_{\mu }$, we have $$K(gv)=g(Kv)=\mu
gv.$$ So $gv\in V_{\mu }$ and $g$ induces a linear transformation on
the nonzero finite-dimensional vector space $V_{\mu}$. There is a
nonzero vector $v'\in V_{\mu}$ such that $gv'=\alpha v'$ for some
nonzero element $\alpha\in \mathbb{K}$. Let $V_{\mu,\alpha}=\{v'\in
V_{\mu}|gv'=\alpha v'\}$. Then $V_{\mu,\alpha}$ is also a nonzero
finite-dimensional linear space. Similarly we can prove that
$h(V_{\mu,\alpha})\subseteq V_{\mu,\alpha}$ as $gh=hg$, $hK$ $=Kh$
by (1.2) and (1.3). Hence there exists a nonzero vector $w\in
V_{\mu,\alpha}$ and $(\mu,\alpha,\beta)$ $\in \mathbbm{K}^{\times
3}$ such that
$$Kw=\mu w,\qquad gw=\alpha w,\qquad hw=\beta w.$$
The proof now follows [6, Proposition VI.3.3], using Lemma 3.1.
\end{proof}

For any positive integer $m$, let
$[m]=\frac{q^{m}-q^{-m}}{q-q^{-1}}$, and $[m]!=[1][2]\cdots[m]$.
Similar to the proof of [6, Lemma VI.3.4], we get the following:
\begin{lemma} Let $v$ be a highest weight vector of weight $(\lambda,\alpha,\beta)$.
Set $v_p=\frac1{[p]!}F^pv$ for $p>0$ and $v_0=v$. Then
$$Kv_p=q^{-2p}\lambda v_p, \qquad gv_p=\alpha  v_p, \qquad Fv_{p-1}=[p]v_p,\qquad hv_{p}=\beta v_p$$
and \begin{align}Ev_p=\frac
{q^{-(p-1)}\lambda-q^{p-1}\lambda^{-1}\alpha^2}{q-q^{-1}}
v_{p-1}.\end{align}
\end{lemma}
\begin{theorem} (a) Let $V$ be a finite-dimensional $U_{g,h}$-module generated by a highest weight
vector $v$ of weight $(\lambda,\alpha,\beta)$. Then

(i) $\lambda=\varepsilon {\alpha}q^n$, where $\varepsilon=\pm 1$ and
$n$ is the integer defined by $dimV=n+1$.

(ii) Setting $v_p=\frac 1 {[p]!}F^pv$, we have $v_p=0$ for $p>n$ and
in addition the set $\{v=v_0,v_1,\cdots, v_n\}$ is a basis of $V$.

(iii) The operator $K$ acting on $V$ is diagonalizable with $(n+1)$
distinct eigenvalues $$\{\varepsilon {\alpha}q^n,\varepsilon
{\alpha}q^{n-2},\cdots,\varepsilon
{\alpha}q^{-n+2},\varepsilon{\alpha} q^{-n}\},$$ and the operators
$g,h$ act on $V$ by scalars $\alpha,\beta$ respectively.

(iv) Any other highest weight vector in $V$ is a scalar multiple of
$v$ and is of weight $(\lambda,\alpha,\beta)$.

(v) The module is simple.

(b) Any simple finite-dimensional $U_{g,h}$-module is generated by a
highest weight vector. Two finite-dimensional $U_{g,h}$-modules
generated by highest weight vectors of the same weight are
isomorphic.
\end{theorem}
\begin{proof} The proof follows that of [6, Theorem VI.3.5] or [13, Theorem
3.4]. It is omitted here.
\end{proof}

Let us denote the $(n+1)$-dimensional simple $U_{g,h}$-module
generated by a highest weight vector $v$ of weight
$(\varepsilon\alpha q^n,\alpha,\beta)$ in Theorem 3.5 by
$V_{\varepsilon,n,\alpha,\beta}$. Since $\mathbbm{K}$ is an
algebraically closed field, the dimension of a simple module over
$\mathbbm{K}[g^{\pm1},h^{\pm1}]$ is equal to one. Any such simple
$\mathbbm{K}[g^{\pm1},h^{\pm1}]$-module is determined by $g\cdot
1=\alpha,  h\cdot 1=\beta,$ for $\alpha,\beta\in \mathbbm{K}^{\times
}$.  This simple module is denoted by
$\mathbbm{K}_{\alpha,\beta}:=\mathbb{K}\cdot1$ in the sequel. The
finite-dimensional simple $U_q(\mathfrak{sl}(2))$-modules are
characterized in [6, Theorem VI.3.5]. These simple modules are
denoted by $V_{\varepsilon,n}$, where $\varepsilon=\pm 1$, and $n\in
\mathbb{Z}_{\geq 0}$. By Proposition 2.2 and [8, Proposition 16.1],
every finite-dimensional simple $U_{g,h}$-module is isomorphic to
$\mathbbm{K}_{\alpha,\beta}\otimes V_{\varepsilon,n}$. It is not
difficult to verify that $\mathbbm{K}_{\alpha,\beta}\otimes
V_{\varepsilon,n}$ is isomorphic to
$V_{\varepsilon,n,\alpha,\beta}$.

\begin{corollary}[Clebsch-Gordan
Formula] Let $n\geq m$ be two non-negative integers. There exists an
isomorphism of $U_{g,h}$-modules
$$V_{\varepsilon,n,\alpha,\beta}\otimes V_{\varepsilon',m,\alpha',\beta'}
\cong
V_{\varepsilon\varepsilon',n+m,\alpha\alpha',\beta\beta'}\oplus
V_{\varepsilon\varepsilon',n+m-2,\alpha\alpha',\beta\beta'}\oplus\cdots\oplus
V_{\varepsilon\varepsilon',n-m,\alpha\alpha',\beta\beta'}.$$
\end{corollary}
\begin{proof} Since $V_{\varepsilon,n,\alpha,\beta}\otimes
V_{\varepsilon',m,\alpha',\beta'}\cong
\mathbbm{K}_{\alpha\alpha',\beta\beta'}\otimes
(V_{\varepsilon,n}\otimes V_{\varepsilon',m})$, and
$$V_{\varepsilon,n}\otimes V_{\varepsilon',m}
\cong V_{\varepsilon\varepsilon',n+m}\oplus
V_{\varepsilon\varepsilon'n+m-2}\oplus\cdots\oplus
V_{\varepsilon\varepsilon',n-m}$$ as modules over
$U_q(\mathfrak{sl}(2))$ by [6, Theorem VII.7.1],
$$V_{\varepsilon,n,\alpha,\beta}\otimes V_{\varepsilon',m,\alpha',\beta'}
\cong
V_{\varepsilon\varepsilon',n+m,\alpha\alpha',\beta\beta'}\oplus
V_{\varepsilon\varepsilon',n+m-2,\alpha\alpha',\beta\beta'}\oplus\cdots\oplus
V_{\varepsilon\varepsilon',n-m,\alpha\alpha',\beta\beta'}.$$ This
completes the proof.\end{proof}

\begin{lemma} Let $m\in \mathbbm{N}$. Then
$$[E,F^m]=[m]F^{m-1}\frac{q^{-(m-1)}K-q^{m-1}K^{-1}g^2}{q-q^{-1}}.$$
\end{lemma}
\begin{proof}Let $E',F',K'$ be the Chevalley generators of
$U_{q}(\mathfrak{sl}(2))$. Then
$$[E',F'^m]=[m]F'^{m-1}\frac{q^{-(m-1)}K'-q^{m-1}K'^{-1}}{q-q^{-1}},$$
by [6, Lemma VI.1.3]. Subsituting $Eg^{-1}$, $g^{-1}K$, $F$ for
$E',K',F'$ in the above identity respectively, we obtain
$$[E,F^m]=[m]F^{m-1}\frac{q^{-(m-1)}K-q^{m-1}K^{-1}g^2}{q-q^{-1}}.$$\end{proof}

Let $M:=\mathbb{K}_{\alpha,\beta}=\mathbb{K}\cdot 1$ be a module
over $\mathbb{K}[g^{\pm1},h^{\pm1}]$, where $g\cdot 1=\alpha$ and
$h\cdot 1=\beta$. About the simple modules over the algebra
$\mathbb{K}[g^{\pm1},h^{\pm1}]$, we have the following
\begin{proposition} Given two simple $\mathbb{K}[g^{\pm1},h^{\pm1}]$-modules
$M:=\mathbb{K}_{\alpha,\beta}$ and $
M':=\mathbb{K}_{\alpha',\beta'},$ if $M$ is not isomorphic to $M'$,
then $\Ext^n(M',M)=0$ for all $n\geq 0$; if $M\cong M'$, then
$$\Ext^n(M',M)\cong\left\{\begin{array}{lr}\mathbb{K},&n=0\\
\mathbb{K}^2,&n=1\\0,& n\geq 2.\end{array}\right.$$
\end{proposition}
\begin{proof} Denote the algebra $\mathbb{K}[g^{\pm1},h^{\pm1}]$ by $R$. Construct a
projective resolution of the simple $R$-module
$\mathbb{K}_{\alpha,\beta}$ as follows:
\begin{align}\xymatrix@C=0.5cm{
  0 \ar[r] & R \ar[r]^{\varphi_2} &R^2 \ar[r]^{\varphi_1} & R \ar[r]^{\varphi_0} & \mathbb{K}_{\alpha,\beta}
   \ar[r] & 0, }\end{align}
where $$\varphi_0(r(g,h))=r(\alpha,\beta),\qquad
   \varphi_1(r(g,h),s(g,h))=r(g,h)(g-\alpha)+s(g,h)(h-\beta),$$ and
   $$\varphi_2(r(g,h))=(r(g,h)(h-\beta),-r(g,h)(g-\alpha))$$ for
   $r(g,h),s(g,h)\in R$. Applying the functor
   $\Hom_R(-,$ $\mathbb{K}_{\alpha,\beta})$ to the exact sequence
   (3.2), we obtain the following complex:
\begin{align}\xymatrix@C=0.5cm{
  0 \ar[r] & \mathbb{K} \ar[r]^{\varphi_1^*} & \mathbb{K}^2 \ar[r]^{\varphi_2^*} & \mathbb{K}
   \ar[r] & 0. }\end{align}
For any $\theta\in \Hom_R(R,\mathbb{K}_{\alpha,\beta})$,
$$\varphi^*_1(\theta)((1,0))=\theta(g-\alpha)=(g-\alpha)\theta(1)=0,$$
$$\varphi^*_1(\theta)((0,1))
=\theta(h-\beta)=(h-\beta)\theta(1)=0.$$ This means that
$\varphi^*_1=0$. Similarly, one can prove that $\varphi_2^*=0$. So
$$\Ext^0(\mathbb{K}_{\alpha,\beta}, \mathbb{K}_{\alpha,\beta} )\cong
\mathbb{K},\qquad \Ext^1(\mathbb{K}_{\alpha,\beta},
\mathbb{K}_{\alpha,\beta} )\cong \mathbb{K}^2,$$
$$\Ext^2(\mathbb{K}_{\alpha,\beta}, \mathbb{K}_{\alpha,\beta} )\cong \mathbb{K},\qquad\Ext^n(\mathbb{K}_{\alpha,\beta},
\mathbb{K}_{\alpha,\beta})=0$$ for $n\geq 3$.

If we use the functor $\Hom_R(-,\mathbb{K}_{\alpha',\beta'})$ to
replace the functor $\Hom_R(-,\mathbb{K}_{\alpha,\beta})$ in the
above proof, we can also obtain the complex (3.3). In this case, we
have
$$\varphi_1^*(\theta)((1,0))=\alpha'-\alpha,\qquad
\varphi_1^*(\theta)((0,1))=\beta'-\beta,$$ and
$$\varphi_2^*(\eta)(a)=(\beta'-\beta)a\eta((1,0))-(\alpha'-\alpha)a\eta((0,1))$$
for $\theta\in \Hom_R(R,\mathbb{K}_{\alpha',\beta'})$,  $\eta\in
\Hom_R(R^2,\mathbb{K}_{\alpha',\beta'})$, and $a\in R$. Hence both
$\varphi_1^*$ and $\varphi_2^*$ are not zero linear mappings
provided that either $\alpha\neq \alpha'$, or $\beta\neq \beta'$.
Consequently, the sequence (3.3) is exact in this case. So $
\Ext^n(\mathbb{K}_{\alpha,\beta}, \mathbb{K}_{\alpha',\beta'})=0$
for $n\geq 0$.
\end{proof}
It is well-known that the group $\Ext^1(M', M)$ can be described by
short exact sequences. Next, we  describe
$\Ext^1(\mathbb{K}_{\alpha,\beta}, \mathbb{K}_{\alpha,\beta})$ by
short exact sequences.

 Let $\xymatrix@C=0.5cm{
  0 \ar[r] & \mathbb{K}_{\alpha,\beta} \ar[r]^{\varphi} & N \ar[r]^{\psi} & \mathbb{K}_{\alpha,\beta}
   \ar[r] & 0}$ be an element in $\Ext^1(\mathbb{K}_{\alpha,\beta},\mathbb{K}_{\alpha,\beta})$. Suppose
$\{w_1,w_2\}$ be a basis of $N$ such that $\psi(w_2)=1$ and
   $w_1=\varphi(1)$. Then $gw_1=\alpha w_1$,
   $hw_1=\beta w_1$. Suppose $gw_2=aw_2+xw_1$. Then $$\alpha\psi(w_2)=\psi(gw_2)=a\psi(w_2).$$ So $gw_2=\alpha
   w_2+xw_1$. Similarly, we can prove $hw_2=\beta w_2+yw_1$.
If $\{u_1,u_2\}$ is another basis satisfying $u_1=\varphi(1)=w_1$
and $\psi(u_2)=1$, then $u_2-w_2=\lambda w_1$ for some $\lambda\in
\mathbb{K}$. Thus $gu_2=\alpha w_2+xw_1+\lambda\alpha w_1=\alpha
u_2+xu_1$. Similarly, we obtain that $hu_2=\beta u_2+yu_1$. Hence
$x,y$ are independent of the choice of the bases of $N$.  So we can
use $M_{x,y}$ to denote the module $N$. In the following, we abuse
notation and use $M_{x,y}$ to denote the following exact sequence
$\xymatrix@C=0.5cm{
  0 \ar[r] & \mathbb{K}_{\alpha,\beta} \ar[r]^{\varphi} & M_{x,y} \ar[r]^{\psi} &\mathbb{K}_{\alpha,\beta}
   \ar[r] & 0}$ meanwhile.

Let $\xymatrix@C=0.5cm{
  0 \ar[r] & \mathbb{K}_{\alpha,\beta} \ar[r]^{\varphi'} & M_{x',y'} \ar[r]^{\psi'} & \mathbb{K}_{\alpha,\beta}
   \ar[r] & 0}$ be another element in $\Ext^1(\mathbb{K}_{\alpha,\beta},\mathbb{K}_{\alpha,\beta})$, and $\{w_1',w_2'\}$ be a basis of $M_{x',y'}$ such that $w_1'=\varphi'(1),$
   $\psi'(w_2')=1$ and
$$gw_1'=\alpha w_1', \qquad gw_2'=\alpha w_2'+x'w_1',$$
$$hw_1'=\beta w_1', \qquad hw_2'=\beta w_2'+y'w_1'.$$
Consider the
   following  commutative diagram
$$\xymatrix{
  0\ar[r]^{ } & \mathbb{K}_{\alpha,\beta} \ar[d]_{\mu_1} \ar[r]^{\varphi } & N \ar[d]_{\mu_2}
  \ar[r]^{\psi } &\mathbb{K}_{\alpha,\beta}\ar[d]_{\mu_3} \ar[r]^{} &0  \\
  0 \ar[r]^{ } &\mathbb{K}_{\alpha,\beta}\ar[r]^{\varphi' } & N' \ar[r]^{\psi'} & \mathbb{K}_{\alpha,\beta} \ar[r]^{} & 0,
  }$$
where $\mu_i$ are isomorphisms. Then
$\mu_2(w_1)=\mu_2(\varphi(1))=\varphi'\mu_1(1)=\mu_1(1)w_1'$.
Suppose $\mu_2(w_2)=aw_1'+bw_2'$. Then
$g\mu_2(w_2)=(a\alpha+bx')w_1'+b\alpha w_2' ,$ and $$\mu_2(gw_2)\\
=\mu_2(\alpha w_2+xw_1)=(a\alpha+x\mu_1(1))w_1'+b\alpha w_2'.$$
Since $g\mu_2(w_2)=\mu_2(gw_2),$  $bx'=x\mu_1(1)$. Similarly, we
have $by'=y\mu_1(1)$. Moreover,
$$\mu_3(1)=\mu_3(\psi(w_2))=\psi'(\mu_2(w_2))=b.$$
If $\mu_1(1)=\mu_3(1)=1$, then $b=1$ and $(x,y)=(x',y')$. Thus
$M_{x,y}= M_{x',y'}$ as elements in the group
$\Ext^1(\mathbb{K}_{\alpha,\beta},\mathbb{K}_{\alpha,\beta})$  if
and only if $(x,y)=(x',y')$.

From Proposition 3.8, we know that
$\Ext^1(\mathbb{K}_{\alpha,\beta},\mathbb{K}_{\alpha,\beta})$ is a
vector space over $\mathbb{K}$. To describe the operations of the
vector space
$\Ext^1(\mathbb{K}_{\alpha,\beta},\mathbb{K}_{\alpha,\beta})$ in the
terms of exact sequences, we use $I$ to denote the ideal of
$R=\mathbb{K}[g^{\pm1},h^{\pm1}]$ generated by $g-\alpha$ and
$h-\beta$, i.e., $I=R(g-\alpha)+R(h-\beta)$. Let $\xi$ be the
embedding homomorphism, and $f$ be the epimorphism of $R$-modules
from $R$ to $\mathbb{K}_{\alpha,\beta}$,  given by
$$f(a(g,h))=a(\alpha,\beta),\quad
a(g,h)\in R.$$ Then we have the following exact sequence of
$R$-modules:
\begin{eqnarray}\xymatrix@C=0.5cm{
  0 \ar[r] & I \ar[r]^{\xi} & R \ar[r]^{f} &
  \mathbb{K}_{\alpha,\beta}
   \ar[r] & 0}.\end{eqnarray}Applying the functor $\Hom_R(-,\mathbb{K}_{\alpha,\beta})$ to the exact sequence
   (3.4) yields the exact sequence
\begin{eqnarray}\xymatrix@C=0.5cm{
  \Hom_R(R,\mathbb{K}_{\alpha,\beta}) \ar[r]^{\tau} & \Hom_R(I,\mathbb{K}_{\alpha,\beta})  \ar[r]^{\partial}&
 \Ext^1( \mathbb{K}_{\alpha,\beta},\mathbb{K}_{\alpha,\beta})
   \ar[r] & 0}.\end{eqnarray}
For any exact sequence of $R$-modules $\xymatrix@C=0.3cm{
  0 \ar[r] & \mathbb{K}_{\alpha,\beta} \ar[r]^{\varphi} & M_{x,y} \ar[r]^{\psi} & \mathbb{K}_{\alpha,\beta}
   \ar[r] & 0,}$ and a basis $\{w_1,w_2\}$ of $M_{x,y}$ satisfying
$\psi(w_2)=1$, $w_1=\varphi(1)$,   define a homomorphism of
$R$-modules $$\sigma: R\rightarrow M_{x,y},\qquad
   \sigma(1)=w_2.$$  Let $\eta_{x,y}$ be a homomorphism of
   $R$-modules
from $I$ to $\mathbb{K}_{\alpha,\beta}$, where
\begin{eqnarray}\eta_{x,y}(a(g,h)(g-\alpha)+b(g,h)(h-\beta))=xa(\alpha,\beta)+yb(\alpha,\beta),\end{eqnarray}
for $a(g,h),b(g,h)\in R$. Now, we have the following commutative
diagram:
$$\xymatrix{
  0\ar[r]^{ } & I \ar[d]_{\eta_{x,y}} \ar[r]^{\xi } & R \ar[d]_{\sigma}
  \ar[r]^{f } & \mathbb{K}_{\alpha,\beta} \ar[d]_{\id} \ar[r]^{} &0  \\
  0 \ar[r]^{ } &\mathbb{K}_{\alpha,\beta}\ar[r]^{\varphi } & M_{x,y} \ar[r]^{\psi} & \mathbb{K}_{\alpha,\beta} \ar[r]^{} &
  0.
  }$$
It is easy to check that $M_{x,y}$ is the pushout of $\eta_{x,y}$
and $\xi$. If we use $M_{kx,ky}$ for any $k\in \mathbb{K}$ to
replace $M_{x,y}$,  we get a homomorphism $\eta_{kx,ky}$ from $I$ to
$\mathbb{K}_{\alpha,\beta}$. Similarly, we have a homomorphism
$\eta_{x+x',y+y'}$ from $I$ to $\mathbb{K}_{\alpha,\beta}$ by using
$M_{x+x',y+y'}$ to replace $M_{x,y}$. From the definition of
$\eta_{x,y}$, one obtains the following:
\begin{eqnarray}\eta_{kx,ky}=k\eta_{x,y},\qquad
\eta_{x+x',y+y'}=\eta_{x,y}+\eta_{x',y'}.\end{eqnarray} Define
$$M_{x,y}\boxplus M_{x',y'}=M_{x+x',y+y'},\quad k\boxdot
M_{x,y}=M_{kx,ky},$$ for $k\in \mathbb{K}$. Then $\{M_{x,y}|x,y\in
\mathbb{K}\}$ becomes a vector space over $\mathbb{K}.$
 By
[12, Theorem 3.4.3], we have a bijection $\Psi_1 $ from
$\{M_{x,y}|x,y\in \mathbb{K}\}$ to
$\Ext^1(\mathbb{K}_{\alpha,\beta},\mathbb{K}_{\alpha,\beta})$ such
that $$\Psi_1(M_{x,y})=\partial (\eta_{x,y})\in
\Ext^1(\mathbb{K}_{\alpha,\beta},\mathbb{K}_{\alpha,\beta}).$$ It
follows from (3.7) that
$$\Psi_1(M_{kx,ky})=k\Psi_1(M_{x,y}),\qquad
\Psi_1(M_{x+x',y+y'})=\Psi_1(M_{x,y})+\Psi_1(M_{x',y'}).$$ Thus
$\Psi_1$ is an isomorphism of vector spaces.

\begin{proposition} Let $V_{\varepsilon,n}$ be a simple $U_q(\mathfrak{sl}(2))$-module with a basis $\{v_0,\cdots,v_n\}$ satisfying
$E'v_0=0, E'v_p=\varepsilon[n-p+1]v_{p-1}$,
$v_p=\frac{F'^p}{[p]!}v_0$, for $p=1,\cdots,n$;  $F'v_n=0$,
$K'v_p=\varepsilon q^{n-2p}v_p$ for $p=0,\cdots,n$, where $E',K',F'$
are Chevalley generators of $U_q(\mathfrak{sl}(2))$. Then
$$V_{\varepsilon,n}\otimes M_{x,y}\in
\Ext^1(V_{\varepsilon,n,\alpha,\beta},V_{\varepsilon,n,\alpha,\beta}),$$
where $M_{x,y}\in
\Ext^1(\mathbb{K}_{\alpha,\beta},\mathbb{K}_{\alpha,\beta})$. The
action of  $U_{g,h}$ on $V_{\varepsilon,n}\otimes M_{x,y}$ with the
basis
$$\{v_0\otimes w_1,\cdots,v_n\otimes w_1;v_0\otimes
w_2,\cdots,v_n\otimes w_2\}$$ is given by
\begin{eqnarray}\left\{\begin{array}{l}E(v_0\otimes w_1)=E(v_0\otimes w_2)=F(v_n\otimes w_1)=F(v_n\otimes
w_2)=0,\\
E(v_p\otimes w_1)=E'v\otimes gw_1=\varepsilon[n-p+1]\alpha
v_{p-1}\otimes w_1,\\
E(v_p\otimes w_2)=\varepsilon\alpha[n-p+1]v_{p-1}\otimes
w_2+\varepsilon[n-p+1]xv_{p-1}\otimes
w_1,\end{array}\right.\end{eqnarray} for $p=1,\cdots,n$;
\begin{eqnarray}v_p\otimes w_i=\frac{F^{p}}{[p]!}v_1\otimes w_i,\quad
F(v_n\otimes w_i)=0\end{eqnarray} for $p=1,\cdots,n$, $i=1,2$;
 \begin{eqnarray}\left\{\begin{array}{l}
 K(v_p\otimes w_1)=K'v_p\otimes gw_1=\varepsilon\alpha q^{n-2p}(v_p\otimes
w_1),\\
K(v_p\otimes w_2)=\varepsilon\alpha q^{n-2p}(v_p\otimes
w_2)+\varepsilon q^{n-2p}x(v_p\otimes
w_1),\end{array}\right.\end{eqnarray} for $p=0,1,\cdots,n$; and
 \begin{eqnarray}\left\{\begin{array}{lr}g(v_p\otimes w_1)=\alpha(v_p\otimes w_1),& g(v_p\otimes
w_2)=\alpha (v_p\otimes w_2)+x(v_p\otimes w_1),\\ h(v_p\otimes
w_1)=\beta(v_p\otimes w_1),& h(v_p\otimes w_2)=\beta (v_p\otimes
w_2)+y(v_p\otimes w_1),\end{array}\right.\end{eqnarray}for
$p=0,1,\cdots,n$. Moreover, $V_{\varepsilon,n}\otimes-$ is an
injective linear mapping from the linear space
$\Ext^1(\mathbb{K}_{\alpha,\beta},\mathbb{K}_{\alpha,\beta})$ to the
linear space
$\Ext^1(V_{\varepsilon,n,\alpha,\beta},V_{\varepsilon,n,\alpha,\beta})$.
\end{proposition}
\begin{proof} We only need to prove that the mapping $V_{\varepsilon,n}\otimes-$ is an injective
linear mapping, since it is easy to check the other results.
Consider the following  commutative diagram
$$\xymatrix{
  0\ar[r]^{ } & V_{\varepsilon,n}\otimes \mathbb{K}_{\alpha,\beta} \ar[d]_{\id} \ar[r]^{\id\otimes\varphi } & V_{\varepsilon,n}\otimes M_{x,y} \ar[d]_{\mu}
  \ar[r]^{\id\otimes\psi } &V_{\varepsilon,n}\otimes \mathbb{K}_{\alpha,\beta} \ar[d]_{\id} \ar[r]^{} &0  \\
  0 \ar[r]^{ } &V_{\varepsilon,n}\otimes \mathbb{K}_{\alpha,\beta}\ar[r]^{\id\otimes\varphi' } & V_{\varepsilon,n}\otimes M_{x',y'} \ar[r]^{\id\otimes\psi'}
  & V_{\varepsilon,n}\otimes \mathbb{K}_{\alpha,\beta} \ar[r]^{} &
  0.
  }$$
Since $(\id\otimes\psi')\mu(v_0\otimes w_2)=(\id\otimes
\psi)(v_0\otimes w_2)=v_0\otimes 1=(\id\otimes \psi')(v_0\otimes
w_2')$,
$$\mu(v_0\otimes w_2)=v_0\otimes w_2'+v\otimes w_1'$$ for some $v\in
V_{\varepsilon,n}$. Then $g\mu(v_0\otimes w_2)=\alpha (v_0\otimes
w_2'+v\otimes w_1')+x'v_0\otimes w_1' ,$ and
$$\mu(g(v_0\otimes w_2))=\mu(\alpha v_0\otimes w_2+xv_0\otimes
w_1)=\alpha(v_0\otimes w_2'+v\otimes w_1')+xv_0\otimes w_1'.$$ Since
$g\mu(v_0\otimes w_2)=\mu(g(v_0\otimes w_2)),$  we have $x=x'$.
Similarly, we can prove that $y=y'$. So $V_{\varepsilon,n}\otimes-$
induces an injective mapping from
$\Ext^1(\mathbb{K}_{\alpha,\beta},\mathbb{K}_{\alpha,\beta})$ to
$\Ext^1(V_{\varepsilon,n,\alpha,\beta},V_{\varepsilon,n,\alpha,\beta})$.

To prove $V_{\varepsilon,n}\otimes -$ is linear, we choose an exact
sequence of
  $U_q(\mathfrak{sl}(2))$-modules
 $$\xymatrix@C=0.5cm{
  0 \ar[r] & L \ar[r]^{} & P \ar[r]^{f} &
  V_{\varepsilon,n}
   \ar[r] & 0,}$$  where $P$ is a finitely generated projective $U_q(\mathfrak{sl}(2))$-module. Then
$P\otimes R$ is a projective $U_{g,h}$-module, and the kernel $Q$ of
$F$ is $\Ker f\otimes R+P\otimes I,$ where $F$ is a homomorphism
from $P\otimes R$ to $V_{\varepsilon,n}\otimes
\mathbb{K}_{\alpha,\beta}$ given by $F(a\otimes b)=f(a)\otimes
b\cdot 1$, $$I=R(g-\alpha)+R(h-\beta).$$ Applying the functor
$\Hom_{U_{g,h}}(-,V_{\varepsilon,n}\otimes
\mathbb{K}_{\alpha,\beta})$ to the exact sequence
$$\xymatrix@C=0.5cm{
  0 \ar[r] & Q \ar[r]^{} & P\otimes R \ar[r]^{F} &
  V_{\varepsilon,n}\otimes \mathbb{K}_{\alpha,\beta}
   \ar[r] & 0}$$ yields an exact sequence
\begin{eqnarray} \qquad  \xymatrix@C=0.4cm{
     \Hom_{U_{g,h}}(U_{g,h},A) \ar[r]^{\tau} & \Hom_{U_{g,h}}(Q,A)
      \ar[r]^{\partial} & \Ext^1(A,A) \ar[r]^{} & 0,}\end{eqnarray}
where $A=V_{\varepsilon,n}\otimes \mathbb{K}_{\alpha,\beta}.$ Define
a homomorphism of $U_{g,h}$-modules $\sigma: P\otimes R\rightarrow
V_{\varepsilon,n}\otimes M_{x,y}$ by
$$\sigma(a\otimes b)=f(a)\otimes b\cdot w_2,\qquad a\otimes b\in
P\otimes R,$$ and a homomorphism of $U_{g,h}$-modules
$\zeta_{x,y}:Q\rightarrow V_{\varepsilon,n}\otimes
\mathbb{K}_{\alpha,\beta}$ by
$$\zeta_{x,y}(a\otimes
b)=\left\{\begin{array}{lll}0,&&a\otimes b\in \Ker f\otimes L,\\
\eta_{x,y}(b)f(a)\otimes 1,&& a\otimes b\in P\otimes
(R(g-\alpha)+R(h-\beta)),\end{array}\right.$$ where $\eta_{x,y}$ is
defined by (3.6). Let $\nu$ be the embedding mapping from $Q$ to
$P\otimes R$. Then we have the following commutative diagram:
$$\xymatrix{
  0\ar[r]^{ } & Q \ar[d]_{\zeta_{x,y}} \ar[r]^{\nu } & P\otimes R \ar[d]_{\sigma}
  \ar[r]^{F } & V_{\varepsilon,n}\otimes\mathbb{K}_{\alpha,\beta} \ar[d]_{\id} \ar[r]^{} &0  \\
  0 \ar[r]^{ } & V_{\varepsilon,n}\otimes\mathbb{K}_{\alpha,\beta}\ar[r]^{1\otimes\varphi }
  & V_{\varepsilon,n}\otimes M_{x,y} \ar[r]^{1\otimes\psi} & V_{\varepsilon,n}\otimes\mathbb{K}_{\alpha,\beta} \ar[r]^{} &
  0.
  }$$It is easy to check that $V_{\varepsilon,n}\otimes M_{x,y}$ is the pushout of $\zeta_{x,y}$
and $\nu$. If we use $M_{kx,ky}$ for any $k\in \mathbb{K}$ to
replace $M_{x,y}$,  we will obtain a homomorphism $\zeta_{kx,ky}$
from $Q$ to $V_{\varepsilon,n}\otimes\mathbb{K}_{\alpha,\beta}$.
Similarly, we  get a homomorphism $\zeta_{x+x',y+y'}$ from $Q$ to
$V_{\varepsilon,n}\otimes\mathbb{K}_{\alpha,\beta}$ by using
$M_{x+x',y+y'}$ to replace $M_{x,y}$. From the definitions of these
mappings and (3.7), we obtain the following
\begin{eqnarray}\zeta_{kx,ky}=k\zeta_{x,y},\qquad
\zeta_{x+x',y+y'}=\zeta_{x,y}+\zeta_{x',y'}.\end{eqnarray} We abuse
notation and write $V_{\varepsilon,n}\otimes M_{x,y}$ for  the
following exact sequence $$\xymatrix{
 0 \ar[r]^{ } & V_{\varepsilon,n}\otimes\mathbb{K}_{\alpha,\beta}\ar[r]^{1\otimes\varphi }
  & V_{\varepsilon,n}\otimes M_{x,y} \ar[r]^{1\otimes\psi} & V_{\varepsilon,n}\otimes\mathbb{K}_{\alpha,\beta} \ar[r]^{} &
  0.  }$$
Define $$(V_{\varepsilon,n}\otimes M_{x,y})\boxplus
(V_{\varepsilon,n}\otimes M_{x',y'})=V_{\varepsilon,n}\otimes
M_{x+x',y+y'}\quad k\boxdot(V_{\varepsilon,n}\otimes
M_{x,y})=V_{\varepsilon,n}\otimes M_{kx,ky},$$ for $k\in
\mathbb{K}.$ Then $\{V_{\varepsilon,n}\otimes M_{x,y}|x,y\in
\mathbb{K}\}$ becomes a vector space with the above operations.
 By
[12, Theorem 3.4.3], we have an injective linear mapping $\Psi_2$ of
linear spaces from
$$\{V_{\varepsilon,n}\otimes M_{x,y}|x,y\in \mathbb{K}\}$$ to
$\Ext^1(V_{\varepsilon,n}\otimes\mathbb{K}_{\alpha,\beta},V_{\varepsilon,n}\otimes\mathbb{K}_{\alpha,\beta})$
such that $$\Psi_2(V_{\varepsilon,n}\otimes M_{x,y})=\partial
(\zeta_{x,y})\in
\Ext^1(V_{\varepsilon,n}\otimes\mathbb{K}_{\alpha,\beta},V_{\varepsilon,n}\otimes\mathbb{K}_{\alpha,\beta}).$$
Therefore,
$$\Psi_2(V_{\varepsilon,n}\otimes M_{kx,ky})=\partial(k\zeta_{x,y})=k\Psi_2(V_{\varepsilon,n}\otimes M_{x,y})$$ and $$
\Psi_2(V_{\varepsilon,n}\otimes
M_{x+x',y+y'})=\Psi_2(V_{\varepsilon,n}\otimes
M_{x,y})+\Psi_2(V_{\varepsilon,n}\otimes M_{x',y'})$$  by (3.13).
Since  $\Psi_2$ is injective and linear,
$$V_{\varepsilon,n}\otimes k\boxdot M_{x,y} =V_{\varepsilon,n}\otimes M_{kx,ky}  =k\boxdot (V_{\varepsilon,n}\otimes M_{x,y})$$ and $$
V_{\varepsilon,n}\otimes( M_{x,y}\boxplus
M_{x',y'})=V_{\varepsilon,n}\otimes
M_{x+x',y+y'}=(V_{\varepsilon,n}\otimes M_{x,y})\boxplus
(V_{\varepsilon,n}\otimes M_{x',y'}).$$ By now, we have completed
the proof.
\end{proof}

We now completely classify all extensions between two
finite-dimensional simple $U_{g,h}$-modules.
\begin{theorem}Suppose $q$ is not a root of unity. Given two simple  $\mathbb{K}[g^{\pm1},h^{\pm1}]$-modules
$\mathbb{K}_{\alpha,\beta}$, $\mathbb{K}_{\alpha',\beta'}$ and a
finite-dimensional simple $U_q(\mathfrak{sl}(2))$-module
$V_{\varepsilon,n}$, the assignment $V_{\varepsilon,n}\otimes -$ is
an isomorphism of vector spaces from
$\Ext^1(\mathbb{K}_{\alpha',\beta'},\mathbb{K}_{\alpha,\beta})$ to
$\Ext^1(M',M)$. Here, $$M:=V_{\varepsilon,n}\otimes
\mathbb{K}_{\alpha,\beta}\cong V_{\varepsilon,n,\alpha,\beta},\qquad
M':=V_{\varepsilon,n}\otimes \mathbb{K}_{\alpha',\beta'}\cong
V_{\varepsilon,n,\alpha',\beta'}.$$ Moreover,
$\Ext^1(V_{\varepsilon,m,\alpha,\beta},V_{\varepsilon',n,\alpha',\beta'})=0$
provided that $(\varepsilon,m,\alpha,\beta)\neq
(\varepsilon',n,\alpha',\beta')$.
\end{theorem}
\begin{proof} Let $C$ be the Casimir element of $U_{g,h}$ defined in Corollary
2.4,
$$d_{m,n}=\frac{q^{m+1}}
{(q^{m-n}-\varepsilon\varepsilon')(q^{m+n+2}-\varepsilon\varepsilon')}
\left(C-\varepsilon'\alpha'\frac
{q^{n+1}+q^{-n-1}}{(q-q^{-1})^2}\right),$$
and$$a=\left\{\begin{array}{ll}\frac{h-\beta^{\prime}}{\beta-\beta^{\prime}},&
if\qquad \beta\neq \beta^{\prime}\\
\\
\frac{g-\alpha^{\prime}}{\alpha-\alpha^{\prime}},&
if\qquad \alpha\neq \alpha'\\
\\
\frac{\varepsilon(q-q^{-1})^2}{\alpha}d_{m,n},&otherwise.\end{array}\right.$$
Then $a$ is in the center of $U_{g,h}$ by Corollary 2.4.

Observe that $V_{\varepsilon,m,\alpha,\beta} \cong
V_{\varepsilon',n,\alpha',\beta'}$ if and only if
$\varepsilon=\varepsilon'$, $\alpha=\alpha'$, $\beta=\beta'$, and
$m=n$.

Suppose $V_{\varepsilon,m,\alpha,\beta}$ is not isomorphic to
$V_{\varepsilon',n,\alpha',\beta'}$, then
$(\varepsilon,m,\alpha,\beta)\neq (\varepsilon',n,\alpha',\beta')$.
 Let $v\in
V_{\varepsilon,m,\alpha,\beta}$ be a nonzero highest weight vector
satisfying
$$Kv=\varepsilon\alpha q^{m}v,\quad gv=\alpha v,\quad hv=\beta v, \quad Ev=0.$$
Then
$$\begin{array}{l}d_{m,n}v=\frac{q^{m+1}}{q^{2m+2}-\varepsilon\varepsilon'q^{m+n+2}-\varepsilon\varepsilon'q^{m-n}+1}(FE+
\frac{qK+q^{-1}K^{-1}g^2}{(q-q^{-1})^2}-\varepsilon'\alpha
\frac{q^{n+1}+q^{-n-1}}{(q-q^{-1})^2})v\\
\\
\qquad\quad=\frac{\varepsilon\alpha}{(q-q^{-1})^2}v,\end{array}$$
 in the case
$\alpha'=\alpha$. Therefore $am=m$ for any $m\in
V_{\varepsilon,m,\alpha,\beta}$ by Schur's Lemma, since
$V_{\varepsilon,m,\alpha,\beta}$ is a simple module and $a$ induces
an endomorphism of $V_{\varepsilon,m,\alpha,\beta}$. Similarly, we
can prove that $am=0$ for any $m\in
V_{\varepsilon',n,\alpha',\beta'}$.

 Consider the short exact sequence of $U_{g,h}$-modules
\begin{eqnarray}0 \xrightarrow [ ]{\qquad}V_{\varepsilon,m,\alpha,\beta} \xrightarrow[ ]{\quad \phi\quad} V
 \xrightarrow[ ]{\quad\varphi\quad} V_{\varepsilon',n,\alpha',\beta'}\xrightarrow[ ]{\qquad}
0.\end{eqnarray} Since $a\varphi(V)=\varphi(aV)=0$,
$$\phi(V_{\varepsilon,m,\alpha,\beta})=\Ker\varphi\supseteq
aV\supseteq
a\phi(V_{\varepsilon,m,\alpha,\beta})=\phi(aV_{\varepsilon,m,\alpha,\beta})
=\phi(V_{\varepsilon,m,\alpha,\beta}).$$ So
$\phi(V_{\varepsilon,m,\alpha,\beta})=aV$. In particular, $a(av)=av$
for any $v\in V$. Therefore $$V= \Ker a\oplus aV=\Ker a\oplus
\phi(V_{\varepsilon,m,\alpha,\beta}).$$ Hence the sequence (3.14) is
splitting and
$\Ext^1(V_{\varepsilon',n,\alpha',\beta'},V_{\varepsilon,m,\alpha,\beta})=0$.

Suppose $(\alpha,\beta)\neq (\alpha',\beta')$. Then $\Ext^1(M',M)=0$
and
$\Ext^1(\mathbb{K}_{\alpha',\beta'},\mathbb{K}_{\alpha,\beta})=0$ by
Proposition 3.8. It is trivial that $V_{\varepsilon,n}\otimes-$ is
an isomorphism of linear spaces.

 Next, we assume that
$V_{\varepsilon,n,\alpha',\beta'}\cong
V_{\varepsilon,m,\alpha,\beta}\cong V_{\varepsilon,n}\otimes
\mathbb{K}_{\alpha,\beta}$. Consider the following exact sequence of
$U_{g,h}$-modules
\begin{eqnarray}0 \xrightarrow [ ]{\qquad}M \xrightarrow[ ]{\quad \phi\quad} V
 \xrightarrow[ ]{\quad\varphi\quad} M\xrightarrow[ ]{\qquad}
0.\end{eqnarray} Since $U_q(\mathfrak{sl}(2))$ is a subalgebra of
$U_{g,h}$, we can regard the exact sequence (3.15) as a sequence of
$U_q(\mathfrak{sl}(2))$-modules. Since every finite-dimensional
$U_q(\mathfrak{sl}(2))$-module is semisimple, there is a
homomorphism $\lambda$ of $U_q(\mathfrak{sl}(2))$-modules from $M$
to $V$ such that $\varphi\lambda=\id_{M}$. For any $v\in V$, we have
$v=(v-\lambda\varphi(v))+\lambda\varphi(v)$. Moreover,
$\varphi(v-\lambda\varphi(v))=0$. Hence $$V=\Ker\varphi\oplus
\im\lambda=\im\phi\oplus \im\lambda,$$ where $\im\lambda\cong
V_{\varepsilon,n}$ as $U_q(\mathfrak{sl}(2))$-modules. Let
$K'=Kg^{-1}$. Suppose $u_1,u_2$ are the highest weight vectors of
the $U_{g,h}$-module $\im\phi$ and the
$U_q(\mathfrak{sl}(2))$-module $\im\lambda$ respectively. Then
$$\{\frac{F^i}{[i]!}u_1, \frac{F^i}{[i]!} u_2|i=0,\cdots,n\}$$ is a
basis of $V$. Moreover,
$$K\varphi(u_2)=g\varphi(K'u_2)=\varepsilon\alpha
q^n\varphi(u_2),\qquad E\varphi(u_2)=g\varphi(Eg^{-1}u_2)=0,$$
$$g\varphi(u_2)=\alpha\varphi(u_2),\qquad
h\varphi(u_2)=\beta\varphi(u_2).$$ So $\varphi(u_2)$ is a highest
weight vector of $M$. Suppose
$gu_2=\sum\limits_{i=0}^na_i\frac1{[i]!}F^iu_1+\sum\limits_{i=0}^nx_i\frac1{[i]!}F^iu_2$.
Then \begin{eqnarray}\qquad \varepsilon
q^ngu_2=gK'u_2=K'gu_2=\varepsilon(\sum\limits_{i=0}^n
q^{n-2i}a_i\frac1{[i]!}
F^iu_1+\sum\limits_{i=0}^nq^{n-2i}x_i\frac1{[i]!}F^iu_2).\end{eqnarray}
Since $q^m\neq 1$ for any positive integer $m$, we obtain
$a_i=x_i=0$, $i=1,2,\cdots,n$ from (3.16).  Hence
$gu_2=a_0u_2+x_0u_1$. Moreover,
$a_0\varphi(u_2)=\varphi(gu_2)=g\varphi(u_2)=\alpha \varphi(u_2).$
So $a_0=\alpha$. Similarly, we can prove that $hu_2=\beta
u_2+y_0u_1$. Moreover, by using Lemma 3.7, one can prove that
\begin{eqnarray}\left\{\begin{array}{l}E(u_1)=E(u_2)=F(\frac1{[n]!}F^n u_1)=F(\frac1{[n]!}F^n
u_2)=0,\\
E(\frac1{[p]!}F^p u_1)=\varepsilon[n-p+1]\alpha F^{p-1} u_1,\\
E(\frac{F^p}{[p]!}
u_2)=\varepsilon\alpha[n-p+1]\frac{F^{p-1}}{[p-1]!}
u_2+\varepsilon[n-p+1]x_0\frac{F^{p-1}}{[p-1]!}u_1,\end{array}\right.\end{eqnarray}
for $p=1,\cdots,n$;
\begin{eqnarray}K(\frac{F^p}{[p]!}u_2)=\varepsilon\alpha q^{n-2p}\frac{F^p}{[p]!}
u_2+\varepsilon q^{n-2p}x_0\frac{F^p}{[p]!} u_1,\end{eqnarray} for
$p=0,1,\cdots,n$; and
\begin{eqnarray}\left\{\begin{array}{ll} g(\frac{F^p}{[p]!} u_1)=\alpha\frac{F^p}{[p]!}u_1,& g(\frac{F^p}{[p]!}
u_2)=\alpha \frac{F^p}{[p]!}u_2+x_0\frac{F^p}{[p]!} u_1,\\
h(\frac{F^p}{[p]!} u_1)=\beta\frac{F^p}{[p]!} u_1,&
h(\frac{F^p}{[p]!} u_2)=\beta \frac{F^p}{[p]!}
u_2+y_0\frac{F^p}{[i]!} u_1,\end{array}\right.\end{eqnarray} for
$p=0,1,\cdots,n$.

Define $\tau(\frac{F^i}{[i]!}u_j)=v_{i}\otimes w_j$ for
$i=0,1,\cdots,n;j=1,2$, and extend it by linearity. Comparing the
relations from (3.8) to (3.11) in Proposition 3.9 with the above
relations from (3.17) to (3.19), we know that $\tau$ is an
isomorphism of $U_{g,h}$-modules from $V$ to
$V_{\varepsilon,n}\otimes M_{x_0,y_0}$. Hence
$V_{\varepsilon,n}\otimes-$ is an isomorphism of linear spaces by
Proposition 3.9.
\end{proof}
\begin{remark}Since $\Ext^1(V_{\varepsilon,n},V_{\varepsilon,n})=0$ and $\Ext^1(V_{\varepsilon,n}\otimes \mathbb{K}_{\alpha,\beta},
V_{\varepsilon,n}\otimes \mathbb{K}_{\alpha,\beta})\neq 0$, the
functor $-\otimes \mathbb{K}_{\alpha,\beta}$ is the zero mapping
from $\Ext^1(V_{\varepsilon,n},V_{\varepsilon,n})$ to
$\Ext^1(V_{\varepsilon,n}\otimes \mathbb{K}_{\alpha,\beta},
V_{\varepsilon,n}\otimes \mathbb{K}_{\alpha,\beta})$. Hence the
functor $-\otimes \mathbb{K}_{\alpha,\beta}$ does not induce an
isomorphism.
\end{remark}

Since $U_{g,h}$ is a Hopf algebra, the dual $M^*$ of any
$U_{g,h}$-module $M$ is still a $U_{g,h}$ module. For $a\in
U_{g,h}$, $f\in M^*$, the action of $a$ on $f$ is given
by$$(af)(m):=f((Sa)m),\qquad m\in M,$$ where $S$ is the antipode of
$U_{g,h}$. Next we describe the dual module of a simple module over
$U_{g,h}$.

\begin{theorem} The dual module $V_{\varepsilon,n,\alpha,\beta}^*$ of the simple $U_{g,h}$-module $V_{\varepsilon,n,\alpha,\beta}$
is a simple module, and $V_{\varepsilon,n,\alpha,\beta}^*\cong
V_{\varepsilon,n,\alpha^{-1},\beta^{-1}}$.\end{theorem}
\begin{proof}  By Theorem 3.5, we can assume that the simple module
$V_{\varepsilon,n,\alpha,\beta}$ has a basis $\{v_0,\cdots,$ $v_n\}$
with relations:
$$Kv_p=\varepsilon q^{n-2p}\alpha v_p, \qquad gv_p=\alpha v_p,\qquad hv_p=\beta
v_p$$ for $p=0,1,\cdots,n,$
$$Fv_{n}=0, \qquad Ev_0=0$$ and
$$Ev_p=\varepsilon\frac
{q^{n-(p-1)}\alpha-q^{p-1-n}\alpha}{q-q^{-1}}v_{p-1}=\varepsilon\alpha[n-p+1]v_{p-1}$$
for $p=1,\cdots,n$. Let $\{v_0^*,\cdots,v_n^*\}$ be the dual basis
of $\{v_0,\cdots,v_n\}$. Then $$(E v_n^*)(v_i)=-v_n^*(EK^{-1}v_i)=-
q^{2i-n}[n-i+1]v_n^*(v_{i-1})=0$$ for $i=1,\cdots,n$, and $$(E
v_n^*)(v_0)=-v_n^*(EK^{-1}v_0)=-\varepsilon\alpha^{-1}
q^{-n}v_n^*(0)=0.$$ Hence $E(v_n^*)=0$. Since
$$(Kv_n^*)(v_i)=v_n^*(K^{-1}v_i)=q^{2i-n}\varepsilon\alpha^{-1}v_n^*(v_i)=\delta_{ni}q^{n}\varepsilon\alpha^{-1}$$
for $i=0,1,\cdots,n$, $Kv_n^*=\varepsilon\alpha^{-1}q^{n}v_n^*$.
Similarly, that $gv^*_n=\alpha^{-1}v^*_n$ follows from
$$(gv_n^*)(v_i)=v_n^*(g^{-1}v_i)=\alpha^{-1} v_n^*(v_i)$$ for
$i=0,1,\cdots,n$, and that $hv^*_n=\beta^{-1}v^*_n$ follows from
$$(hv_n^*)(v_i)=v_n^*(h^{-1}v_i)=\beta^{-1}v_n^*(v_i)$$ for $i=0,1,\cdots,n$.
So $V_{\varepsilon,n,\alpha,\beta}^*$ is a simple $U_{g,h}$-module
generated by the highest weight vector $v_n^*$ with weight
$(\varepsilon \alpha^{-1} q^n, \alpha^{-1},\beta^{-1})$. Hence
$V_{\varepsilon,n,\alpha,\beta}^*\cong
V_{\varepsilon,n,\alpha^{-1},\beta^{-1}}$.\end{proof}

Let $H$ be a Hopf algebra, and $H^{\circ}=\{f\in H^*|\ker f$
contains an ideal $I$ such that the dimension of $H/I$ is
finite$\}$. Then $H^{\circ}$ is a Hopf algebra, which is called the
finite dual Hopf algebra of $H$. Now let $M$ be a left module over
the Hopf algebra $U_{g,h}$. For any $f\in M^*$ and $v\in M$, define
a coordinate function $c_{f,v}^M\in U_{g,h}^*$ via
$$c_{f,v}^M(x)=f(xv)\qquad\qquad for\quad x\in H.$$
If $M$ is finite dimensional, then $c_{f,v}^M\in U_{g,h}^{\circ}$,
the finite dual Hopf algebra  of $U_{g,h}$. The coordinate space
${C}(M)$ of $M$ is a linear subspace of $U_{g,h}^*$, spanned by the
coordinate functions $c_{f,v}^M$ as $f$ runs over $M^*$ and $v$ over
$M$.

\begin{corollary}Let $A$ be the subalgebra of $U_{g,h}^{\circ}$ generated by all the coordinate functions of all
finite dimensional simple $U_{g,h}$-modules. Then $A$ is a sub-Hopf
algebra of $U_{g,h}^{\circ}$.\end{corollary}

\begin{proof}Let $\hat{\mathcal{C}}$ be the subcategory of the left $U_{g,h}$-module category consisting of
all finite direct sums of finite dimensional simple
$U_{g,h}$-modules. Then $\hat{\mathcal{C}}$ is closed under tensor
products and duals by Corollary 3.6 and Theorem 3.12. Thus $A$ is a
sub-Hopf algebra of $U^{\circ}_{g,h}$  and  is the directed union of
the coordinate spaces ${C}(V)$ for $V\in\hat{\mathcal {C}}$ by [2,
Corollary I.7.4].\end{proof}

Finally, we describe the simple modules over $U_{g,h}$ when $q$  is
a root of unity.  Assume that the order  of  $q$ is $d>2$ and define
$$e=\left\{\begin{array}{ll}d,&if\ d\ is\ odd,\\
\frac d2,&otherwise.\end{array}\right.$$  We will use the notations
$V(\lambda,a,b)$, $V(\lambda,a,0)$, $\widetilde{V}(\pm q^{1-j},c)$
to denote finite-dimensional simple $U_q(\mathfrak{sl}(2))$-modules.
These simple modules have been described in [6, Theorem VI.5.5]. The
next results follow from [6, Proposition VI.5.1, Proposition VI.5.2,
Theorem VI.5.5] and [8, Proposition 16.1].
\begin{proposition} Suppose $q$ is a root of unity. Then

(1) Any simple $U_{g,h}$-module of dimension $e$ is isomorphic to a
module of the following list:

(i) $\mathbbm{K}_{\alpha,\beta}\otimes V(\lambda,a,b)$, where
$\mathbbm{K}_{\alpha,\beta}=\mathbb{K}\cdot1$ is a one-dimensional
module over $\mathbbm{K}[g^{\pm 1},h^{\pm 1}]$, and $g\cdot
1=\alpha,$ $ h\cdot 1=\beta$ for some $\alpha,\beta\in
\mathbbm{K}^{\times }$.

(ii) $\mathbbm{K}_{\alpha,\beta}\otimes V(\lambda,a,0)$, where
$\lambda$ is not of the form $\pm q^{j-1}$ for any $1\leq j\leq
e-1$,

(iii) $\mathbbm{K}_{\alpha,\beta}\otimes\widetilde{V}(\pm q^{1-j},c)
$.

(2) Any simple $U_{g,h}$-module of dimension $n<e-1$ is isomorphic
to a module of the form $V_{\varepsilon,n,\alpha,\beta}$, where the
structure of $V_{\varepsilon,n,\alpha,\beta}$ is given by Theorem
3.5.

(3) The dimension of any simple $U_{g,h}$-module is not larger than
 $e$.
\end{proposition}

\section{Verma modules and the category $\mathcal {O}$}

In this section, we assume that the nonzero element $q\in
\mathbb{K}$ is not a root of unity. We will study the BGG
subcategory of the category of all left $U_{g,h}$-modules. For the
undefined terms in this section, we refer the reader to [8] and
[10].

If $M$ is a $U_{g,h}$-module, a maximal weight vector is any nonzero
$m\in M$ that is killed by $E$, and is a common eigenvector for
$K,g,h$. A standard cyclic module is one which is generated by
exactly one maximal weight vector. For each $(a,b,c)\in
{\mathbbm{K}}^{\times 3}$, define the Verma module
$$V(a,b,c):=U_{g,h}/I(a,b,c),$$ where $I(a,b,c)$ is the left ideal of $U_{g,h}$
generated by $E$, $K-a$, $g-b,$ $h-c$. $V(a,b,c)$ is a free
${\mathbbm{K}}[F]$-module of rank one, by the PBW Theorem 2.1 for
$U_{g,h}$. Hence the set $W(V(a,b,c))$ of weights of the Verma
module $V(a,b,c)$ is equal to $\{ (q^{-2n}a,b,c)|n\geq 0\}$.

About the extension group $\Ext^1(V(a',b',c'),V(a,b,c))$ of two
Verma modules $V(a',b',c'),$ $V(a,b,c)$, we have the following:
\begin{proposition}{\it Suppose $V(a,b,c)$ and $V(a',b',c')$ are two Verma modules. Then $\Ext^1((V(a,b,c),V(a,b,c))\neq0$
and $\Ext^1((V(a',b',c'),V(a,b,c))=0$ if $a,b,c;a',b',c'$ satisfy
one of the following conditions.

(1) $(b,c)\neq (b',c')$;

(2) $(b,c)=(b',c')$,  $a\neq a'$ and $aa'\neq q^{-2}b^2$.}
\end{proposition}
\begin{proof}Let $M_{x,y}\in
\Ext^1(\mathbb{K}_{b,c},\mathbb{K}_{b,c})$ be the module described
in Proposition 3.9, where either $x\neq 0$ or $y\neq 0$. Consider
the $U_{g,h}$-module $M=V(ab^{-1})\otimes M_{x,y}$, where
$V(ab^{-1})$ is a Verma module over $U_q(\mathfrak{sl}(2))$
generated by a highest weight vector $v$ with weight $ab^{-1}$.
Suppose $w_1,w_2$ is a basis of $M_{x,y}$ such that $gw_1=bw_1$,
$gw_2=bw_2+xw_1$, $hw_1=cw_1$, $hw_2=cw_2+yw_1$. Then $K(v\otimes
w_1)=a(v\otimes w_1)$ and $$K(v\otimes w_2)=a(v\otimes
w_2)+ab^{-1}x(v\otimes w_1).$$Therefore the subspace $V_1$ of $M$
generated by
$$\frac{F^n}{[n]!}v\otimes w_1, \qquad n\in \mathbb{Z}_{\geq0}$$ is
a $U_{g,h}$-module, which is isomorphic to $ V(a,b,c)$. Moreover
$M/V_1$ is also isomorphic to $ V(a,b,c)$. Thus $M\in
\Ext^1(V(a,b,c),V(a,b,c))$. Suppose $M\cong V(a,b,c)\oplus
V(a,b,c)$. Then the actions of $g,h$ on $M$ are given via
multiplications by $b,c$ respectively. This is impossible when
either $x\neq 0$, or $y\neq 0$. So $M$ is a nonzero element in
$\Ext^1(V(a,b,c),V(a,b,c))$.

Now let $$u=\left\{\begin{array}{ll}\frac{g-b'}{b-b'},&if\  b\neq b'\\
\frac{h-c'}{c-c'},&if\ c\neq c'\\
\frac{aa'(q-q^{-1})^2}{(a-a')(qaa'-q^{-1}b^2)}(C-\frac{qa'+q^{-1}a'^{-1}b^2}{(q-q^{-1})^2}),&if\
a\neq a',aa'\neq q^{-2}b^2,(b,c)=(b',c'),\end{array}\right.$$ where
$C$, which is given in Corollary 2.4, is the Casimir element of
$U_{g,h}$. Then $u$ is in the center of $U_{g,h}$ by Corollary 2.4.
Suppose $V(a,b,c)$ and $V(a',b',c')$ are generated by the highest
weight vectors $v,v'$ respectively. It is easy to check that $uv=v$
and $uv'=0$. So $u$ induces the identity endomorphism of $V(a,b,c)$
and the zero endomorphism of $V(a',b',c')$. Similar to the proof of
Theorem 3.10, we can prove that every short exact sequence
$$0\rightarrow V(a,b,c)\rightarrow N\rightarrow
V(a',b',c')\rightarrow 0$$is splitting. Hence
$\Ext^1((V(a',b',c'),V(a,b,c))=0$.
\end{proof}
\begin{remark}It is unknown whether
$\Ext^1(V(q^{-2}a^{-1}b^2,b,c),V(a,b,c))=0$ in the case when
$b^2\neq q^2a^2$.\end{remark} The proof of the following proposition
is standard (see e.g. [5], [7] or [8]).
\begin{proposition}{\it (1) The Verma module $V(a,b,c)$ has a unique maximal submodule $N(a,b,c)$, and the quotient $V(a,b,c)/N(a,b,c)$
is a simple module $L(a,b,c)$.

(2) Any standard cyclic module is a quotient of some Verma
module.}\end{proposition}

By [8, Theorem 4.2] and Proposition 2.2, every Verma module over
$U_{g,h}$ is isomorphic to $V(\lambda)\otimes \mathbb{K}_{b,c}$,
where $V(\lambda)$ is a Verma module over $U_q(\mathfrak{sl}(2))$.
Conversely, $V(\lambda)\otimes \mathbb{K}_{b,c}$ is a Verma
$U_{g,h}$ module if  $V(\lambda)$ is a Verma module over
$U_q(\mathfrak{sl}(2))$. In the following, we determine when the
Verma module $V(\lambda)\otimes \mathbb{K}_{b,c}$ is isomorphic to
the Verma module $V(a,b,c)$, using the isomorphism in Proposition
2.2(1).

\begin{proposition}{\it Suppose $V(\lambda)$ is a Verma module over $U_q(\mathfrak{sl}(2))$ and $\mathbb{K}_{b,c}$ is a simple module
over $\mathbb{K}[g^{\pm1},h^{\pm1}]$. Then
$V(\lambda)\otimes\mathbb{K}_{b,c}$ is a Verma module over $U_{g,h}$
with the highest weight $(b\lambda,b,c)$. Conversely, every Verma
module $V(a,b,c)$ over $U_{g,h}$ is isomorphic to
$$V(ab^{-1})\otimes \mathbb{K}_{b,c},$$ where $V(ab^{-1})$ is a
Verma module over $U_q(\mathfrak{sl}(2))$.

Therefore the Verma module $V(a,b,c)$ is isomorphic to
$V(\lambda)\otimes \mathbb{K}_{b',c'}$ if and only if
$(a,b,c)=(b'\lambda,b',c')$.}
\end{proposition}

\begin{proof}Suppose $E',K',F'$ are Chevalley generators
of $U_q(\mathfrak{sl}(2))$. Let $V(\lambda)$ be a Verma module over
$U_q(\mathfrak{sl}(2))$. Then $V(\lambda)$ has a basis $\{v_p|p\in
\mathbb{Z}_{\geq 0}\}$ satisfying
$$K'v_p=\lambda q^{-2p}v_p,\qquad K'^{-1}v_p=\lambda^{-1}q^{2p}v_p,$$
$$E'v_{p+1}=\frac{q^{-p}\lambda-q^p\lambda^{-1}}{q-q^{-1}}v_p,\qquad
F'v_p=[p+1]v_{p+1}$$ and $E'v_0=0$. Since $U_{g,h}\cong
U_q(\mathfrak{sl}(2))\otimes$ $ \mathbb{K}[g^{\pm 1},h^{\pm1}]$,
then $V(\lambda)\otimes \mathbb{K}_{b,c}$ is a cyclic module with
the highest vector $v_0\otimes 1$, where the action of $x\otimes
y\in U_q(\mathfrak{sl}(2))\otimes \mathbb{K}[g^{\pm 1},h^{\pm1}]$ on
$v\otimes 1\in V(\lambda)\otimes \mathbb{K}_{b,c}$ is given by
$$(x\otimes y)\cdot(v\otimes 1)=x\cdot v\otimes y\cdot 1.$$
The highest weight of $V(\lambda)\otimes \mathbb{K}_{b,c}$ is
$(b\lambda,b,c)$. Let $v=1+I(b\lambda,b,c)$ be the highest weight
vector of the Verma module $V(b\lambda,b,c)$. Define a linear map
$f$ from $V(\lambda)\otimes \mathbb{K}_{b,c}$ to $V(a,b,c)$ by
$f(v_p\otimes 1)=\frac1{[p]!}F^pv$. Similar to [6, Proposition
VI.3.7], we can prove that $f$ is a homomorphism of
$U_{g,h}$-modules. Therefore $V(\lambda)\otimes \mathbb{K}_{b,c}$ is
the Verma module with highest weight $(b\lambda,b,c)$ by Proposition
4.3(2).

Conversely, let $\lambda=ab^{-1}$. Consider an infinite-dimensional
vector space $V(\lambda)$ with basis $\{v_i|i\in \mathbb{Z}_{\geq
0}\}$. For $p\geq0$, set
$$K'v_p=\lambda q^{-2p}v_p,\qquad K'^{-1}v_p=\lambda^{-1}q^{2p}v_p,$$
$$E'v_{p+1}=\frac{q^{-p}\lambda-q^p\lambda^{-1}}{q-q^{-1}}v_p,\qquad
F'v_p=[p+1]v_{p+1}$$ and $E'v_0=0$, where $E',K',F'$ are Chevalley
generators of $U_q(\mathfrak{sl}(2))$. Then $V(\lambda)$ is a Verma
module over $U_q(\mathfrak{sl}(2))$ with the above actions by [6,
Lemma VI.3.6]. The highest weight of $V(\lambda)\otimes
\mathbb{K}_{b,c}$ is $(a,b,c)$.  Therefore $V(\lambda)\otimes
\mathbb{K}_{b,c}$ is isomorphic to the Verma module over $U_{g,h}$
with highest weight $(a,b,c)$.
\end{proof}

One of the basic questions about a Verma module is to determine its
maximal weight vectors. We now answer this question.
\begin{theorem} Let $V(a,b,c), V(a',b',c')$  be two Verma modules, where $a,b,c;a',b',c'$ $\in$ ${\mathbbm{K}}^{\times}$.

(1) If $V(a,b,c)$ has a maximal weight vector of weight $(
 q^{-2n}a,b,c)$, then it is unique up to scalars and $a=\varepsilon bq^{n-1}$ with $n>0$.

 (2) $\dim_{\mathbbm{K}}\Hom_{U_{g,h}}(V(a',b',c'),V(a,b,c))=0$ or $1$ for all
 $(a',b',c')$ and $(a,b,c)$, and all nonzero homomorphisms between
 two Verma modules are injective.

 (3) The nonzero submodule of $V(a,b,c)$ (which is unique if it exists) is precisely of the form
 $$V(q^{-2n}a,b,c)={\mathbbm{K}}[F]v_{q^{-2n}a,b,c}.$$
 \end{theorem}
\begin{proof}Suppose $p(F)=(a_nF^n+a_{n-1}F^{n-1}+\cdots+a_0)\bar1$ is a maximal weight vector, where
$\bar1$ is the maximal weight vector of $V(a,b,c)$ and $a_n\neq 0$.
Then
$$E(p(F))=[n]\frac{q^{-n+1}a-q^{n-1}a^{-1}b^2}{q-q^{-1}}a_nF^{n-1}\bar1+(lower\
degree\ terms)\bar 1=0,$$ by Lemma 3.7. This implies $a=\varepsilon
bq^{n-1}.$ Moreover $p(X)=a_nF^n\bar1$ and (1) follows.

(2) follows from (1) and the fact that ${\mathbbm{K}}[F]$ is a
principal ideal domain directly.

If $M$ is a nonzero submodule of $V(a,b,c)$, then $M$ contains a
vector of the highest possible weight $(q^{-2n}a,b,c)$. We claim
that $M=V(q^{-2n}a,b,c)={\mathbbm{K}}[F]v_{q^{-2n}a,b,c}$, where
$v_{q^{-2n}a,b,c}$ is the weight vector in $M$ with weight
$(q^{-2n}a,b,c)$. The weight vector $v_{q^{-2n}a,b,c}$ is unique up
to scalar by (1). To prove the above claim, we only need to show
that $M\subseteq{\mathbbm{K}}[F]v_{q^{-2n}a,b,c}$.

Suppose, to the contrary, that $v\in M$ is of the form
$$v=p(F)v_{q^{-2n}a,b,c}+a_{n-1}F^{n-1}\bar1+\cdots+a_1F\bar1+a_0\bar1.$$
We may assume that $p(F)=0$ because $v_{q^{-2n}a,b,c}\in M$. Since
$K^iv\in M$ for any $i$, $a_{n-k}F^{n-k}\bar1\in M$,
$k=1,2,\cdots,n$. This is a contradiction since $(q^{-2i}a,b,c)$ is
not a weight of $M$ if $i<n$.
\end{proof}
\begin{remark}(1) If $\frac a b\neq \varepsilon q^{n}$ for any $n\geq 1$, then the Verma module is a simple module by Theorem 4.5.

(2) It is well-known that the Verma module $V(\lambda)$ over
$U_q(\mathfrak{sl}(2))$ is simple provided that $\lambda\neq
\varepsilon q^{n}$ for any integer $n>0$, where $\varepsilon=\pm1$.
Since $$V(\lambda)\otimes \mathbb{K}_{b,c}\cong V(b{\lambda},b,c)$$
by Proposition 4.4, $V(\lambda)\otimes \mathbb{K}_{b,c}$ is a simple
$U_{g,h}$-module provided that $\lambda\neq \varepsilon q^n$ for any
$n$, where $\varepsilon=\pm 1$.

(3) The simple module $L(a,b,c)$ is finite-dimensional if and only
if the only maximal submodule $N(a,b,c)$  of $V(a,b,c)$ is equal to
$V(q^{-2n}a,b,c)$ and $a=\varepsilon bq^{n-1}$ for some $n\in
\mathbb{N}$. In this case, $L(\varepsilon bq^{n-1},b,c)\cong
V_{\varepsilon,n-1,b,c}$, which is given by Theorem 3.5.\end{remark}

Finally, we study the BGG category $\mathcal {O}$, which is defined
below.
\begin{definition} The BGG category $\mathcal{O}$ consists of all finitely generated $U_{g,h}$-modules and all homomorphisms of modules with the following
properties:

(1) The actions of $K,g,h$ are diagonalized with finite-dimensional
weight spaces.

(2) The $B_+$-action is locally finite, where $B_+$ is the
subalgebra generated by $E,$ $K^{\pm1},$ $g^{\pm1},$ $h^{\pm1}$.
\end{definition}

It is obvious that every Verma module is in $\mathcal{O}$. By
Theorem 3.5, all finite-dimensional simple $U_{g,h}$-modules are in
$\mathcal {O}$. Any simple module in $\mathcal{O}$ is isomorphic to
either a simple Verma module or a finite-dimensional simple module
$V_{\varepsilon,n,\alpha,\beta}$ described in Theorem 3.5. In fact,
if $M$ is a simple module in $\mathcal{O}$, then $M=U_{g,h}v$ for
some common eigenvector $v$ of $K,g,h$. Suppose $Kv=\lambda v$. Then
$KE^nv=q^{2n}\lambda E^nv$ for any positive integer $n$. Since the
action of $E$ is locally finite, there is an $n$ such that $E^nv\neq
0$ and $E^{n+1}v=0$. Thus $M=U_{g,h}E^nv$ is a standard cyclic
$U_{g,h}$-module. So it is a quotient of a Verma module. Hence it is
isomorphic to either a simple Verma module or a finite-dimensional
simple module $V_{\varepsilon,n,\alpha,\beta}$.

Suppose $\xymatrix@C=0.3cm{
  0 \ar[r] & V_{\varepsilon,n,\alpha,\beta} \ar[r]^{} &M \ar[r]^{} & V_{\varepsilon,n,\alpha,\beta}\ar[r] & 0 }$ is a nonzero element in
$\Ext^1(V_{\varepsilon,n,\alpha,\beta},
V_{\varepsilon,n,\alpha,\beta})$.  We remark that this $M$ is not in
$\mathcal {O}$ since the actions of $g,h$ on $M$ can not be
diagonalized by Proposition 3.8 and Theorem 3.10. Similarly, if
$\xymatrix@C=0.3cm{
  0 \ar[r] & V(a,b,c) \ar[r]^{} &M \ar[r]^{} & V(a,b,c)\ar[r] & 0 }$
  is
a nonzero element in $\Ext^1(V(a,b,c),V(a,b,c))$, then
 $M$  is not in
$\mathcal {O}$.

By using results in [8], we obtain that every finite-dimensional
module in $\mathcal{O}$ is semisimple. In the following we give a
direct proof of this fact.
\begin{proposition}Every finite-dimensional module in $\mathcal{O}$
is semisimple.
\end{proposition}
\begin{proof}Let $0=M_0\subseteq M_1\subseteq \cdots\subseteq
M_n=M$ be a composition series of a finite-dimensional module $M$
for $M\in \mathcal {O}$. We prove that $M$ is semisimple by using
induction. If $n=2$, then we have the following exact sequence
$$0\rightarrow M_1\rightarrow M\rightarrow
M/M_1\rightarrow 0.$$ Suppose the above sequence is not splitting,
then the either the action of $g$ or the action of $h$ on $M$ is not
semisimple by Theorem 3.10. Thus $M\notin \mathcal{O}$. This
contradiction implies that  $M$ is semisimple. Suppose $M$ is
semisimple in the case when $n=k\geq 2$. Now let $n=k+1$. Then
$M_k=\oplus_{i=1}^kS_i$ is a direct sum of simple $U_{g,h}$-modules
$S_i$ by the assumption. Now let $N_i=S_1\oplus\cdots\oplus
\widehat{S_{i}}\oplus \cdots\oplus S_k$, where $\widehat{S_{i}}$
means that $S_i$ is omitted. Consider the following commutative
diagrams for $i=1,2,\cdots,k$:
$$\xymatrix{
  0\ar[r]^{ } & M_k \ar[d]_{\lambda_i} \ar[r]^{\phi } & M \ar[d]_{\pi_i}
  \ar[r]^{ \pi} & M/M_k \ar[d]_{\id} \ar[r]^{} &0  \\
  0 \ar[r]^{ } &N_i\ar[r]^{\varphi_i } & M/S_i \ar[r]^{\psi_i} & M/M_k \ar[r]^{} & 0,
  }$$where $\phi,\varphi_i$ are embedding mappings, and
  $\lambda_i,\pi_i,\pi,\psi_i$ are the canonical projections.
Since the bottom exact sequences are splitting by the inductive
assumption, there are homomorphisms $\xi_i:M/S_i\rightarrow $ $N_i$
such that $\xi_i\varphi_i=\id_{N_i}.$ Define $\xi:M\rightarrow M_k$
via
$$\xi(m)=\frac1{k-1}\sum\limits_{i=1}^k\xi_i\pi_i(m)$$ for $m\in M$.
Now let $m=m_1+\cdots+m_k\in M_k$, where $m_i\in S_i$. Then
$$\xi_i\pi_i(m)=\xi_i\pi_i\phi(m)=\xi_i\varphi_i\lambda_i(m)=m-m_i, $$ and
$$\xi\phi(m)=\frac1{k-1}\sum\limits_{i=1}^k\xi_i\pi_i(m)=m.$$
This means that the top exact sequence of the above commutative
diagrams is splitting. Hence $M\cong M_k\oplus \Ker \xi\cong
M_k\oplus M/M_k\cong S_1\oplus\cdots\oplus S_k\oplus M/M_k$ is
semisimple.
\end{proof}

By the PBW Theorem 2.1, the algebra $U_{g,h}$ has a triangular
decomposition $\mathbb{K}[F]\otimes H\otimes $ $\mathbb{K}[E]$,
where $H=\mathbb{K}[K^{\pm1},g^{\pm1},h^{\pm1}]$. In the same way as
[8, Definition 11.1], we can define the Harish-Chandra projection
$\xi$ as follows:
$$\xi:=\varepsilon\otimes \id\otimes
\varepsilon:U_{g,h}=\mathbb{K}[F]\otimes H\otimes
\mathbb{K}[E]\rightarrow H.$$ Let $V(a,b,c)$ be a Verma module
generated by a nonzero highest weight vector $v$. Then
\begin{eqnarray}Cv=\frac{qa+q^{-1}a^{-1}b^2}{(q-q^{-1})^2}v,\qquad gv=bv,\qquad hv=cv,\end{eqnarray}
where $C$ is the Casimir element of $U_{g,h}$. By Corollary 2.4, the
center of $U_{g,h}$ is $\mathbb{K}[C,g^{\pm1},h^{\pm1}]$. For any
element $z\in\mathbb{K}[C,g^{\pm1},h^{\pm1}]$, $zv=\xi_{(a,b,c)}(z)
v$ for some $\xi_{(a,b,c)}(z)\in \mathbb{K}$. Then
$$\xi_{(a,b,c)}\in
\Hom_{alg}(\mathbb{K}[C,g^{\pm1},h^{\pm1}],\mathbb{K}).$$ We call
$\xi_{(a,b,c)}$  the central character determined by $V(a,b,c)$.

\begin{proposition}{\it (1) Suppose $V(a,b,c)$ and $V(a',b',c')$ are two Verma modules. Then $\xi_{(a',b',c')}=\xi_{(a,b,c)}$
if and only if \begin{eqnarray} (a-a')(aa'-q^{-2}b^2)=0,\qquad
b=b',\qquad c=c'.
\end{eqnarray}
(2) $\Hom_{U_{g,h}}(V(a,b,c),V(a',b',c'))\neq 0$ if and only if
$a=\varepsilon q^{-n-1}b$ and $a'=$ $\varepsilon q^{n-1}b$ for some
nonnegative integer $n$ and $(b,c)=(b',c')$}.
\end{proposition}
\begin{proof}Let $v,v'$ be the nonzero highest weight vectors of
$V(a,b,c)$ and $V(a',b',c')$ respectively. Then
$\xi_{(a',b',c')}=\xi_{(a,b,c)}$ if and only if
$Cv'=\xi_{(a,b,c)}(C)v'$, $gv'=\xi_{(a,b,c)}(g)v'$ and
$hv'=\xi_{(a,b,c)}(h)v'$. Thus (4.2) follows from (4.1).

If there is a nonzero homomorphism $\varphi$ from $V(a,b,c)$ to
$V(a',b',c')$, then $$\xi_{(a',b',c')}=\xi_{(a,b,c)}.$$ Thus (4.2)
holds. Suppose $\varphi(v)=(\sum\limits_{i=0}^na_iF^i)v'$, where
$a_{n}\neq 0$. Since $\varphi(Kv)=K\varphi(v)$,
$$aa_i=q^{-2i}a_ia'$$ for $i=0,1,\cdots,n$. Hence $a=q^{-2n}a'$ and $a_i=0$ for $0\leq i\leq
n-1$. Observe that
$$0=\varphi(Ev)=E\varphi(v)=a_nEF^nv'=a_n[n]\frac{a'q^{-n+1}-a'^{-1}q^{n-1}b^2}{q-q^{-1}}F^{n-1}v'.$$
Hence $aa'=q^{2n-2}b^2$. So $a'=\varepsilon q^{n-1}b$ and
$a=\varepsilon q^{-n-1}b$.

Conversely, notice that $V(a,b,c)=\mathbb{K}[F]v$ and
$V(a',b',c')=\mathbb{K}[F]v'$ are two free $\mathbb{K}[F]$-modules.
Thus the mapping
$$\varphi(f(F)v)=f(F)F^nv',\qquad  f(F)\in \mathbb{K}[F]$$ is a nonzero
linear mapping. Since $b=b'$ and $c=c'$,
$\varphi(gf(F)v)=g\varphi(f(F)v)$ and
$\varphi(hf(F)v)=h\varphi(f(F)v)$. It is routine to check that
$\varphi(Ef(F)v)=E\varphi(f(F)v)$ and
$\varphi(Kf(F)v)=K\varphi(f(F)v)$. So $\varphi$ is a nonzero
homomorphism of $U_{g,h}$-modules.
\end{proof}

For any  $\nu\in
\Hom_{alg}(\mathbb{K}[C,g^{\pm1},h^{\pm1}],\mathbb{K})$, define a
full subcategory $\mathcal {O}(\nu)$ of $ \mathcal {O}$ as follows:
$$\mathcal {O}(\nu)=\{M\in \mathcal {O}|\forall m\in
M,z\in\mathbb{K}[C,g^{\pm1},h^{\pm1}], \exists n\in \mathbb{N}\text{
such that } (z-\nu(z))^nm=0\}.$$ For any $\nu\in
\Hom_{alg}(\mathbb{K}[C,g^{\pm1},h^{\pm1}],\mathbb{K})$, suppose
$\nu(C)=\mu$, $\nu(g)=b$, $\nu(h)=c$. Then $b,c\in
\mathbb{K}^{\times}$. Since $\mathbb{K}$ is an algebraically closed
field, there is $a\in \mathbb{K}$ such that
$\frac{qa+q^{-1}a^{-1}b^2}{(q-q^{-1})^2}=\mu$. Therefore the Verma
module $V(a,b,c)\in \mathcal{O}(\nu)$ by (4.1), and
$\mathcal{O}(\nu)$ is not empty. By results in [8, Theorem 11.2], we
have the following decomposition of $\mathcal {O}$.

\begin{theorem}The category $\mathcal {O}=\bigoplus\limits_{\nu\in \Hom_{alg}(\mathbb{K}[C,g^{\pm1},h^{\pm1}],\mathbb{K})}\mathcal {O}(\nu).$

\end{theorem}

Let $\mathcal {H}$ be the Harish-Chandra category over
$(U_{g,h},H)$, which consists of all $U_{g,h}$-modules $M$ with a
simultaneous weight space decomposition for
$H=\mathbb{K}[K^{\pm1},g^{\pm1},h^{\pm1}]$, and finite-dimensional
weight spaces. By Proposition 2.5, $U_{g,h}$ has an anti-involution
$i$. Thus we can define a duality functor $F:\mathcal {H}\rightarrow
\mathcal {H}$ as follows: $F(M)$ is the vector space  spanned by all
$ {H}$-weight vectors in $M^*=\Hom_{\mathbb{K}}(M,\mathbb{K}).$ It
is a module under the action determined by $$\langle
am^*,m\rangle=\langle m^*,i(a)m\rangle$$ for $a\in U_{g,h}$, $m^*\in
F(M)$, $m\in M$. By results in [8], $F$ defines a duality functor
$F:\mathcal {O}\rightarrow \mathcal {O}^{op}$. Moreover,
$F(L(a,b,c))=L(a,b,c)$, $F(V(a,b,c))$ has the socle $L(a,b,c)$ and
so on.

By Proposition 4.9,  $U_{g,h}$ satisfies the condition (S4) defined
in [8]. Therefore it satisfies the conditions (S1), (S2), and (S3)
by [8, Proposition 11.3] and [8, Theorem 10.1], where (S1), (S2) and
(S3) are defined in [8]. By [8, Theorem 4.3], we have the following
theorem since $\Gamma$ is trivial.

\begin{theorem}{\it
Let $\nu\in
\Hom_{\mathbb{K}}(\mathbb{K}[K^{\pm1},g^{\pm1},h^{\pm}],\mathbb{K})$
and $\mathcal {O}(\nu)$ have the same meaning as in Theorem 4.10.
Then:

(1) Each object of the  block $\mathcal {O}(\nu)$ has a filtration
whose subquotients are quotients of Verma modules.

(2) Each block $\mathcal {O}(\nu)$ has enough projective objects.

(3) Each block $\mathcal {O}(\nu)$ is a highest weight category,
equivalent to the category of finitely generated right modules over
a finite-dimensional $\mathbb{K}$-algebra.

In particular, BGG Reciprocity holds in $\mathcal {O}$.}
\end{theorem}

\end{document}